\documentclass{article}
\usepackage{amsmath}
\usepackage{amsthm}
\usepackage{amssymb}
\usepackage[dvips]{graphicx}
\usepackage{latexsym}
\usepackage{mathtools}
\usepackage{geometry}
 \renewcommand{\equation}

\newtheorem{thm}{Theorem}[section]

\newtheorem{corresp*}{Correspondence}
\newtheorem{facts*}{Facts}
\newtheorem{claim}{Claim}[section]

\theoremstyle{definition}
\newtheorem{rmk}{Remark}[section]
\newtheorem{defini}{Definition}[section]
\newtheorem{obs}{Observation}[section]

\title{The Goeritz group of a Heegaard splitting of genus two of a Seifert manifold  whose base orbifold is sphere with three exceptional points of sufficiently complex coefficients }
\author{Nozomu Sekino}
\date{}

\begin{document}
\maketitle

\begin{abstract}
In this paper, we add examples to Goeritz groups, the mapping class groups of given Heegaard splittings of 3-manifolds. 
We focus on a Heegaard splitting of genus two of a Seifert manifold  whose base orbifold is sphere with three exceptional points of sufficiently complex coefficients, where ``sufficiently complex'' means that every surgery coefficient $\frac{p_{l}}{q_{l}}$ of each exceptional fiber (in a surgery description) satisfies $q_{l}\not \equiv \pm 1$ mod $p_l$.
\end{abstract}

\section{Introduction}
It is known that every connected orientable closed 3-manifold $X$ admits a {\it Heegaard splitting} $X=V^{+}\cup_{\Sigma}V^{-}$: 
This represents a connected orientable closed surface $\Sigma$ in $X$ and two handlebodies $V^{+}$ and $V^{-}$ bounded by $\Sigma$ in $X$ ``inner'' and ``outer'', respectively. 
The {\it Hempel distance} \cite{hempel} of a Heegaard splitting $X=V^{+}\cup_{\Sigma}V^{-}$ is defined to be the distance between the sets of the boundaries of properly embedded essential disks of $V^{+}$ and $V^{-}$ in (the 1-skelton of) the curve complex of $\Sigma$, and this measures some complexity of the Heegaard splitting. 

For a Heegaard splitting $X=V^{+}\cup_{\Sigma}V^{-}$, the set of isotopy classes of orientation preserving self-homeomorphisms of $X$ fixing $V^{+}$ (so also $V^{-}$ and $\Sigma$), where $\Sigma$ is fixed as a set during isotopies, forms a group by compositions. 
This group is called the {\it Goeritz group} of $X=V^{+}\cup_{\Sigma}V^{-}$. 
In other words, the Goeritz group of $X=V^{+}\cup_{\Sigma}V^{-}$ is a subgroup of the mapping class group of $\Sigma$ consisting of the elements whose representatives can extend to self-homeomorphisms of both of $V^{+}$ and $V^{-}$. 

About Goeritz groups, Johnson \cite{johnson1} showed that if the Hempel distance of a Heegaard splitting is at least $4$, then the Goeritz group of the Heegaard splitting is finite. 
This means that a high Hempel distance Heegaard splitting is very ``rigid''. 
On the other hand, the Goeritz group of a Heegaard splitting has an element of infinite order if the Hempel distance of the Heegaard splitting is at most $1$ (, twist along a reducing sphere or so-called ``eye glass twist'' using weakly reducible pair). 
In this case, determining the structure of the Goeritz group is difficult in general. 
It is unknown whether the Goeritz group of the Heegaard spliting of genus greater than $3$ of the $3$-sphere is finitely generated or not. 
By the sequence of works \cite{goeritz},\cite{scharlemann},\cite{cho1},\cite{cho2},\cite{ck1},\cite{ck2},\cite{ck3},\cite{ck4}, 
a finite presentation of the Goeritz group of every Heegaard splitting of genus two of Hempel distance at most $1$ is known. 
Freedman and Scharlemann \cite{fs} gave a finite generating set of the Goeritz group of the Heegaard splitting of genus three of $3-$sphere. 
For the case where Hempel distance is $2$ or $3$, the Goeritz groups can be finite or infinite. 
For examples of infinite cases, the Goeritz group of a Heegaard splitting induced by an open book decomposition of some type  \cite{johnson2} or by a twisted book decomposition of some type \cite{ik1} is infinite, these examples have Hempel distance $2$. 
For examples of finite cases, the Goeritz group of a Heegaard splitting having some ``keen''-type property, which represents the rigidity of a Heegaard splitting \cite{ik1} is finite, these examples have Hempel distance at least $2$. 
About finitely generativity, the Goeritz group of a Heegaard splitting having ``thick isotopy'' property and finiteness of the mapping class group of the ambient manifold \cite{iguchi} is finitely generated, these examples have Hempel distance at least $2$. 
The example of a infinite Goeritz group of a Heegaard splitting of Hempel distance $3$ is not found for now. 

In this paper, we will add some concrete examples of the Goeritz groups of some Heegaard splittings of 3-manifolds of some types, a Heegaard splitting of genus two of a Seifert manifold  whose base orbifold is sphere with three exceptional points of sufficiently complex coefficients. 
The manifold we consider has a surgery description as in Figure \ref{surg_desc} and $M$ denotes the manifold. 
The word ``with three exceptional points of sufficiently complex coefficients'' means that $q_{l}\not \equiv \pm 1$ mod $p_l$ for all $l=1,2,3$, and we always assume that $p_l>0$. 

\begin{figure}[htbp]
 \begin{center}
  \includegraphics[width=40mm]{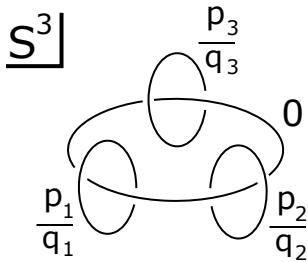}
 \end{center}
 \caption{A surgery description of $M$ with $q_{i}\not \equiv \pm 1$ mod $p_i$ for all $i=1,2,3$}
 \label{surg_desc}
\end{figure}

Every Heegaard splitting of genus two is of the form $H_{\{i,j\}}$ for $i\neq j \in \{1,2,3\}$ (defined in Subsection \ref{hij}). 
We will compute the Goeritz group of $H_{\{i,j\}}$, denoted by $\mathcal{G}(H_{\{i,j\}})$ as follows, where $h$ and $\iota$ are defined in Subsection \ref{invs}:
\begin{thm} \label{main}
\begin{itemize}
\item $\mathcal{G}(H_{\{i,j\}})= \left< h| h^2 \right>\cong \mathbb{Z}/2\mathbb{Z}$ if ``$p_{i} \neq p_{j}$'' or ``$p_{i}=p_{j}$ and $q_{i} \not \equiv q_{j}$ mod $p_i$''.
\item $\mathcal{G}(H_{\{i,j\}})= \left< h,\ \iota| h^2,\ \iota^{2}, \ h\iota h \iota \right>\cong (\mathbb{Z}/2\mathbb{Z})\oplus (\mathbb{Z}/2\mathbb{Z})$ if  $p_{i}=p_{j}$ and $q_{i} \equiv q_{j}$ mod $p_i$.
\end{itemize}
\end{thm}

The rest of this paper is organized as follows: 
In Section \ref{gg_hs}, we give the definition of Goeritz groups, determine Heegaard splittings of genus two of $M$ and their diagrams, and give two elements of the Goeritz groups. 
In Section \ref{t_o}, we give terminologies and observations about simple closed curves on oriented surface of genus two, used for the proof. 
In Section \ref{computation}, we give the proof. \\

Throughout the paper, curves in surfaces are assumed to intersect minimally and essentially otherwise stated. 
We often identify curves in surfaces and their isotopy classes. 
For a submanifold $B$ in manifold $A$, $N(B;A)$ denotes a sufficiently small regular neighborhood of $B$ in $A$. 
Sometimes, $A$ is abbreviated if $A$ is clear from the context. 
For a finite set $N$, $|N|$ denotes the number of elements of $N$. 
For two elements $a,b$, the ordered pair is denoted by $(a,b)$ and the unordered pair is denoted by $\{a,b\}$. 

\section*{Acknowledgements}
The author would like to thank professor Takuya Sakasai for supporting his research, and he also would like to thank professor Sangbum Cho and professor Yuya Koda for introducing him the area of Goeritz groups.

\section{Goeritz groups and genus two Heegaard splittings of $M$}\label{gg_hs}
\subsection{Goeritz groups}
We give a definition of the Goeritz group of a given Heegaard splitting.\\ 
Let $X$ be a connected, closed and orientable 3-manifold, and $X=V^{+}\cup_{\Sigma}V^{-}$ a Heegaard splitting of $X$. 
\begin{defini}(The Goeritz group of $X=V^{+}\cup_{\Sigma}V^{-}$)\\
The Goeritz group of $X=V^{+}\cup_{\Sigma}V^{-}$ is a group whose elements are the isotopy classes relative to $\Sigma$ of self homeomorphisms $f$ of $X$ satisfying $f(V^{+})=V^{+}$ and whose products are given by compositions. 
\end{defini}
In another way, the Goeritz group is a subgroup of the mapping class group of the splitting surface $\Sigma$ consisting of the elements which extend to self homeomorphisms of both of $V^{+}$ and $V^{-}$. 

\subsection{Genus two Heegaard splitting of $M$} \label{hs}
$M$ has a surgery description as in Figure \ref{surg_desc}. 
It is known that $M$ is an irreducible, non-exceptional Seifert manifold (see \cite{jaco} for example). 
It is also known that every Heegaard splitting of such manifolds is irreducible, and every irreducible Heegaard splitting of Seifert manifolds is either vertical or horizontal \cite{ms}. 
It can be computed that $M$ has no genus two horizontal Heegaard splittings under $q_{i}\not \equiv  \pm1$ mod $p_{l}$ for all $l=1,2,3$. 
A Seifert manifold which has genus two horizontal Heegaard splitting must have a genus one fibered knot whose monodromy is periodic (as boundary free map). 
And such Seifert manifold has an exceptional fiber of coefficient $\frac{p}{q}$ for $p=2,3,6$ since the order of every periodic orientation preserving self-homeomorphism of torus with one boundary component is $1,2,3,$ or $6$. 
There is no $q$ satisfying $q\not \equiv  \pm1$ mod $p$ for such $p$. 
Thus all Heegaard splittings of $M$ are vertical. We will see them in the next subsection. 

\subsection{Vertical Heegaard splittings of $M$} \label{hij}
A vertical Heegaard splitting of $M$ is obtained as follows: 
Choose $i\neq j\in\{1,2,3\}$. 
Take a disk $D$ which is a Seifert surface of the $0$-framed unknot in Figure \ref{surg_desc} such that it intersects every $\frac{p_l}{q_l}$-framed knot, $K_l$ once for $l=1,2,3$. 
Connect $K_i$ and $K_j$ by a simple arc $\alpha$ on $D$. 
Set $W=N(K_{i}\cup \alpha \cup K_{j})$, which is a handlebody of genus two. 
In fact, $V=M\setminus W$ is also a handlebody of genys two. 
A Heegaard splitting $M=V\cup_{s}W$ is called a vertical Heegaard splitting of $M$, where $S$ denotes the splitting surface. 
A schematic of this splitting is depicted in Figure \ref{vert_hs}. 
This splitting is independent of the choice of $\alpha$ up to homeomorphisms. 
Thus there are at most three Heegaard splitting of $M$ up to homeomorphisms, choices of $i\neq j\in \{1,2,3\}$. 
This Heegaard splitting is denoted by $H_{\{i,j\}}$.

\begin{figure}[htbp]
 \begin{center}
  \includegraphics[width=60mm]{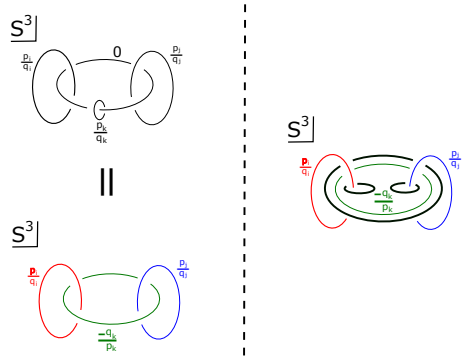}
 \end{center}
 \caption{A genus two vertical Heegaard splitting of $M$}
 \label{vert_hs}
\end{figure}

\subsection{Standard diagrams of $H_{\{i,j\}}$}
There are standard diagrams for $H_{\{i,j\}}$ as in Figure \ref{std_diag}. 
In this figure, boxes with rational numbers (need not to be reduced) imply the curves in the boxes: For a box with number $\frac{a}{b}$, draw horizontal $|a|$-lines and vertical $|b|$-lines. Then resolve the intersection points so that curves (we follow from the bottom of the box) twist toward left if $\frac{a}{b}>0$, and curves (we follow from the bottom of the box) twist toward right if $\frac{a}{b}<0$. 
Examples are depicted in Figure \ref{sign}. We allow reducible rational numbers in boxes. 
By pasting 2-handles along green and purple curves inner the surface, we get a handlebody $V$, and by pasting 2-handles along red and blue curves outer the surface, we get a handlebody $W$. We call these disks $D',E',D_{L},D_{R},E$ and give them orientations as in Figure \ref{std_diag}. 
Note that $D_{R}$ is obtained from $E$ and $D_{L}$ by band-sums, see Figure \ref{band_sum}. 
Note that these diagrams may be non minimally intersecting. 
By this diagrams, we can see that the Hempel distance of $H_{\{i,j\}}$ is less than $3$ since there is a curve on $S$ which is disjoint from both of $\partial D'$ and $\partial E$. 
Furthermore, that is grater than $1$ since $M$ is irreducible and the genus of $S$ is two. 
Thus the Hempel distance of $H_{\{i,j\}}$ is $2$.

\begin{figure}[htbp]
 \begin{center}
  \includegraphics[width=140mm]{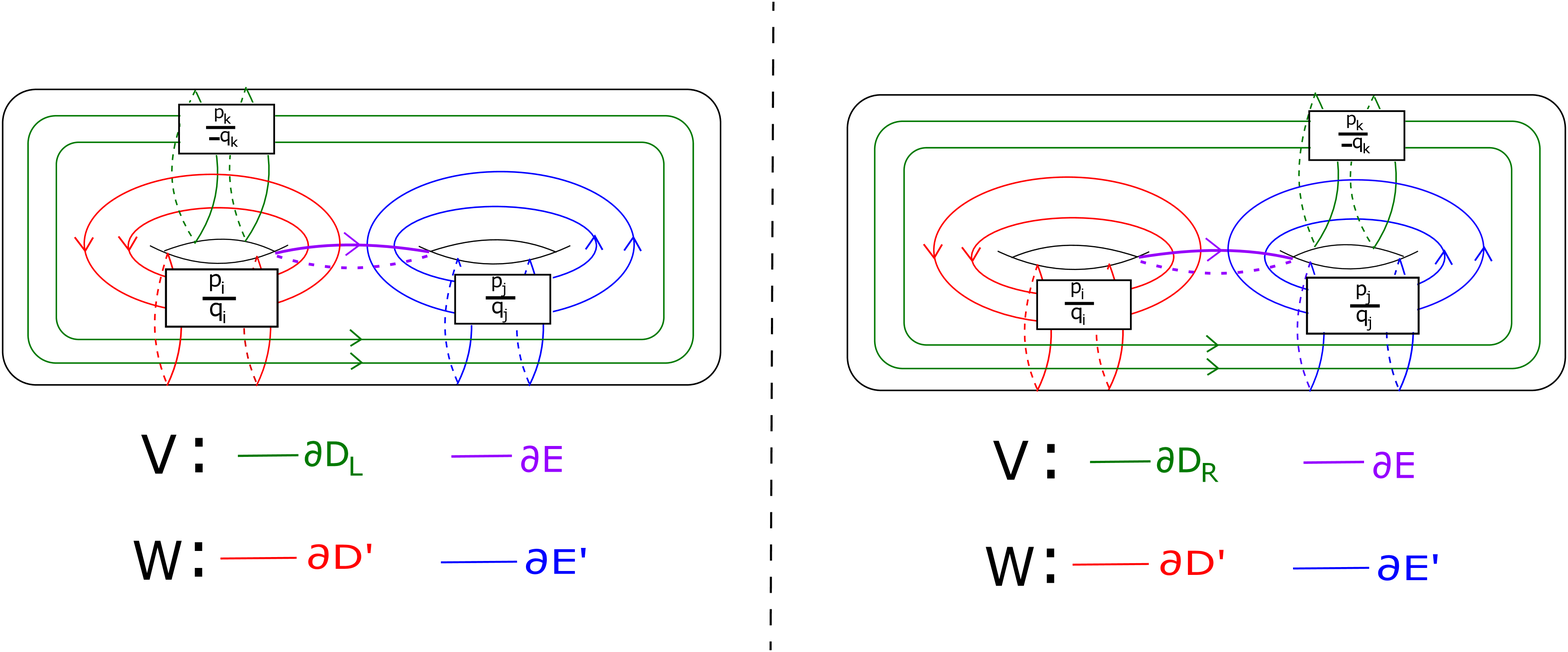}
 \end{center}
 \caption{(left) Left standard diagram\ \ \ \ (right) Right standard diagram}
 \label{std_diag}
\end{figure}

\begin{figure}[htbp]
 \begin{center}
  \includegraphics[width=60mm]{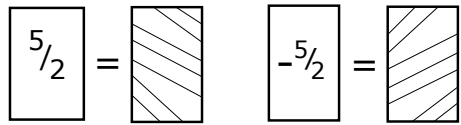}
 \end{center}
 \caption{Notations of boxes with numbers}
 \label{sign}
\end{figure}

\subsection{Useful standard diagrams of $H_{\{i,j\}}$}
Moreover, we replace standard diagrams with more useful ones: 
There is unique numbers $q'_{i}$ and $q'_{j}$ such that $q'_{i}\equiv q_{i}$ mod $p_i$, $2|q'_{i}|<p_i$, and $q'_{j}\equiv q_{j}$ mod $p_j$, $2|q'_{j}|<p_j$ hold. 
By applying Dehn twists and band sums as in Figures \ref{dehn_surg}, \ref{band_sum}, we have useful standard diagrams, which are obtained from the standard diagrams in Figure \ref{std_diag} by replacing  $q_i$,$q_j$,$q_k$ with $q'_i$,$q'_j$,$q'_k$. Note that $q'_{k}\equiv q_{k}$ mod $p_k$ holds and that $2|q'_{k}|<p_k$ does not necessarily hold. 
We continue to call the disks in useful standard diagrams $D',E',D_{L},D_{R},E$. 
Note that under $p_{i}=p_{j}$, $q'_{i}=q'_{j}$ if and only if $q_{i}\equiv q_{j}$ mod $p_{i}$. 
These diagrams may be non minimally intersecting. 
In subsection \ref{minimallyintersect}, we list diagrams after making minimally intersecting.

\begin{figure}[htbp]
 \begin{center}
  \includegraphics[width=100mm]{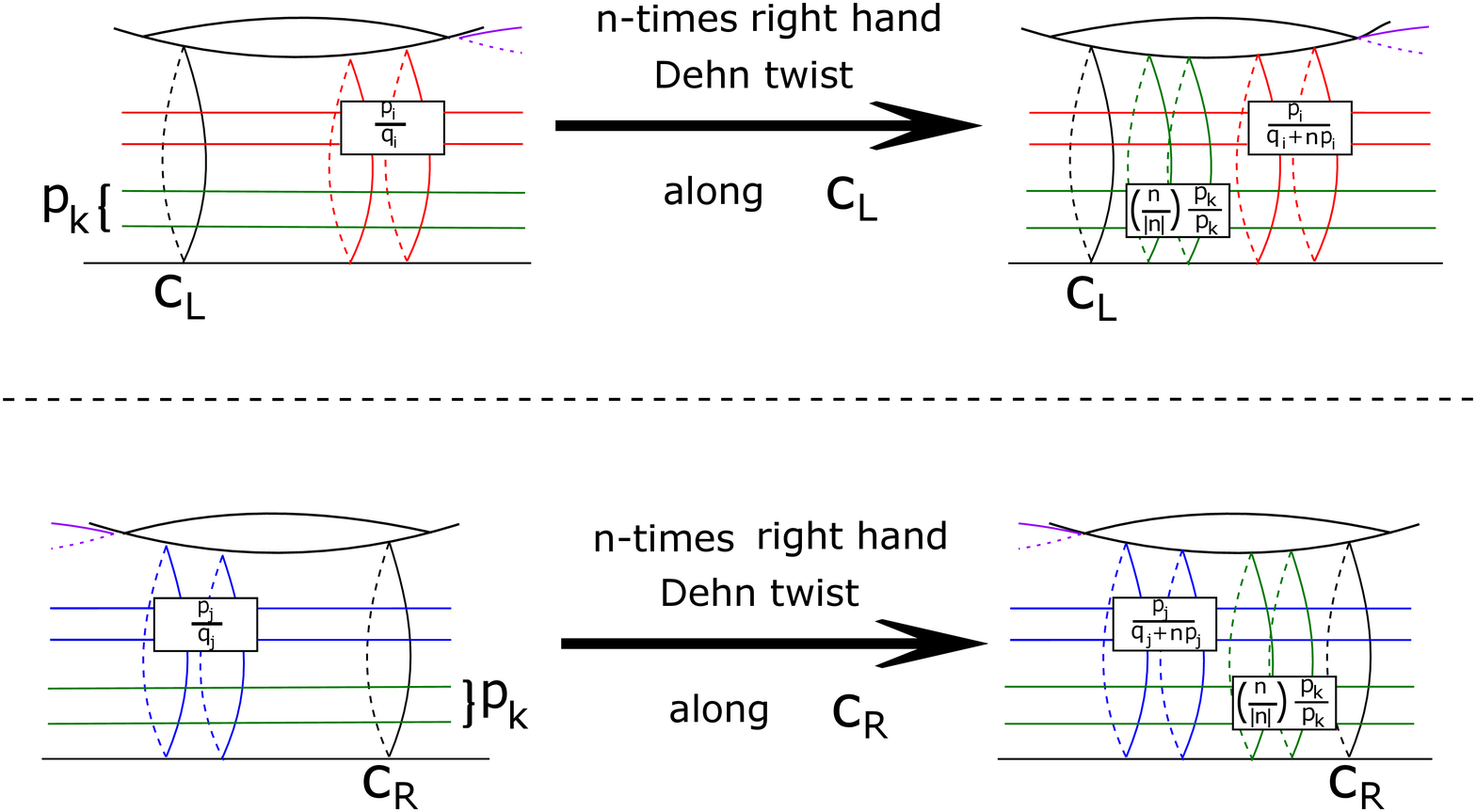}
 \end{center}
 \caption{Dehn twists}
 \label{dehn_surg}
\end{figure}

\begin{figure}[htbp]
 \begin{center}
  \includegraphics[width=100mm]{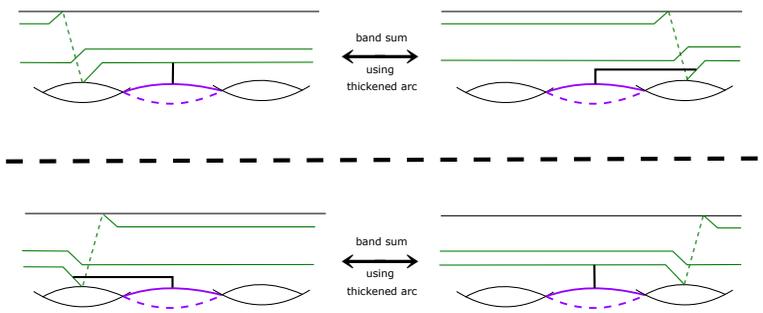}
 \end{center}
 \caption{Band sums}
 \label{band_sum}
\end{figure}

\subsection{Two elements of the Goeritz group of $H_{\{i,j\}}$} \label{invs}
There are two elements of the Goeritz group of $H_{\{i,j\}}$. We will define them. 
We work on useful standard diagrams. 

\subsubsection{An involution $h$} 
An involution $h$ is defined as in Figure \ref{inv_h}, a $\pi$-rotation along the horizontal axis $l_{h}$ (as a self-homeomorphism of $S$). 
This $h$ maps $D'$,$E'$,$D_{L}$,$E$ to themselves, reversing the orientations. 
Hence $h$ extends to self-homeomorphisms of both of $V$ and $W$. 
Thus $h$ is indeed an element of the Goeritz group. 
\begin{figure}[htbp]
 \begin{center}
  \includegraphics[width=100mm]{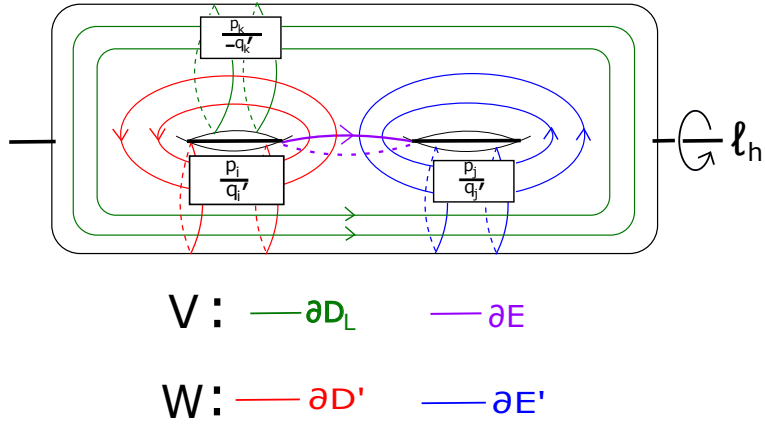}
 \end{center}
 \caption{An involution $h$}
 \label{inv_h}
\end{figure}

\subsubsection{An involution $\iota$} 
We assume that $p_i=p_j$ and $q_{i}\equiv q_{j}$ mod $p_{i}$. 
This implies that $q'_{i}=q'_{j}$. 
An involution $\iota$ is defined as in Figure \ref{inv_iota}, a $\pi$-rotation along the vertical axis $l_{\iota}$ (as a self-homeomorphism of $S$). 
This $\iota$ exchanges $D'$ and $E'$, fixes $E$, and maps $D_{L}$ to $D_{R}$. 
Hence $\iota$ extends to self-homeomorphisms of both of $V$ and $W$. 
Thus $\iota$ is indeed an element of the Goeritz group. 
\begin{figure}[htbp]
 \begin{center}
  \includegraphics[width=80mm]{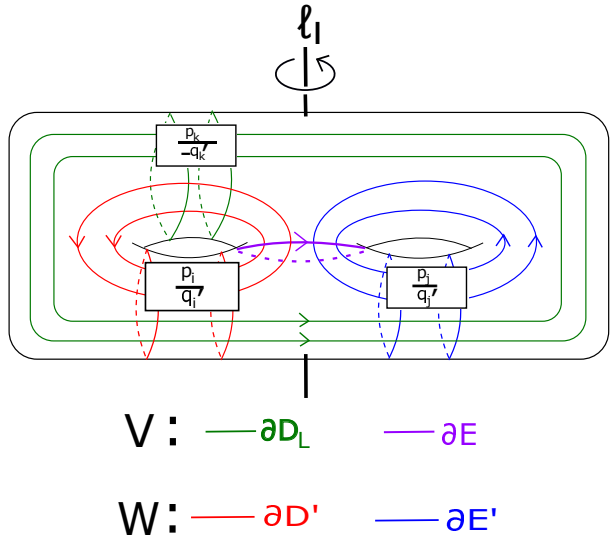}
 \end{center}
 \caption{An involution $\iota$}
 \label{inv_iota}
\end{figure}

\subsection{Minimally intersecting standard diagrams of $H_{\{i,j\}}$}\label{minimallyintersect}
We list minimally intersecting useful standard diagram of $H_{\{i,j\}}$. 
When $q'_{i}q'_{k}>0$, the diagram obtained from the left of Figure \ref{std_diag} by replacing $q_l$ with $q'_{l}$ for all $l=1,2,3$ is minimally intersecting. 
And when $q'_{j}q'_{k}>0$, the diagram obtained from the right of Figure \ref{std_diag} by replacing $q_l$ with $q'_{l}$ for all $l=1,2,3$ is minimally intersecting. 
The other cases are as in Figures \ref{left_qi+_pk},\ref{left_qi+_qk},\ref{left_qi-_pk},\ref{left_qi-_qk},\ref{right_qj+_pk},\ref{right_qj+_qk},\ref{right_qj-_pk},\ref{right_qj-_qk}. 

In terms of the planar diagrams (defined in Definition \ref{planar}), we see the following for any cases: 
\begin{itemize}
\item The diagram of $\partial D' \cup \partial E'$ on $\mathcal{P}_{\{ \partial D_{L}, \partial E\}}$ has $p_{j}-|q'_j|$ consecutive parallel oriented arcs coming from $\partial E'$ connecting $\partial E^{+}$ and $\partial E^{-}$,
 $p_{i}-|q'_{i}|$ consecutive parallel oriented arcs coming from $\partial D'$ connecting $\partial D^{+}_{L}$ and $\partial D^{-}_{L}$,
 and $(p_{k}-1)$ blocks of $|q'_{j}|$ consecutive parallel oriented arcs coming from $\partial E'$ connecting $\partial D^{+}_{L}$ and $\partial D^{-}_{L}$. 

\item The diagram of $\partial D_{L} \cup \partial E$ on $\mathcal{P}_{\{ \partial D', \partial E'\}}$ has $2$ consecutive parallel oriented arcs coming from $\partial E$ connecting $\partial E'^{+}$ and $\partial E'^{-}$ (since $p_{j}>2|q'_j|$),
 $|p_k|$ consecutive parallel oriented arcs coming from $\partial D_{L}$ connecting $\partial E'^{+}$ and $\partial E'^{-}$.

\item The diagram of $\partial D' \cup \partial E'$ on $\mathcal{P}_{\{ \partial D_{R}, \partial E\}}$ has $p_{i}-|q'_i|$ consecutive parallel oriented arcs coming from $\partial D'$ connecting $\partial E^{+}$ and $\partial E^{-}$,
 $p_{j}-|q'_{j}|$ consecutive parallel oriented arcs coming from $\partial E'$ connecting $\partial D^{+}_{R}$ and $\partial D^{-}_{R}$,
 and $(p_{k}-1)$ blocks of $|q'_{i}|$ consecutive parallel oriented arcs coming from $\partial D'$ connecting $\partial D^{+}_{R}$ and $\partial D^{-}_{R}$. 

\item The diagram of $\partial D_{R} \cup \partial E$ on $\mathcal{P}_{\{ \partial D', \partial E'\}}$ has $2$ consecutive parallel oriented arcs coming from $\partial E$ connecting $\partial D'^{+}$ and $\partial D'^{-}$ (since $p_{i}>2|q'_i|$),
 $|p_k|$ consecutive parallel oriented arcs coming from $\partial D_{R}$ connecting $\partial D'^{+}$ and $\partial D'^{-}$.
\end{itemize}
Note that $|q'_l|\geq 2$ and $p_l \geq 5$ for all $l=i,j,k$ and $p_{l}-|q'_{l}|>|q'_{l}|$ for $l=i,j$.  

\begin{figure}[htbp]
 \begin{center}
  \includegraphics[width=90mm]{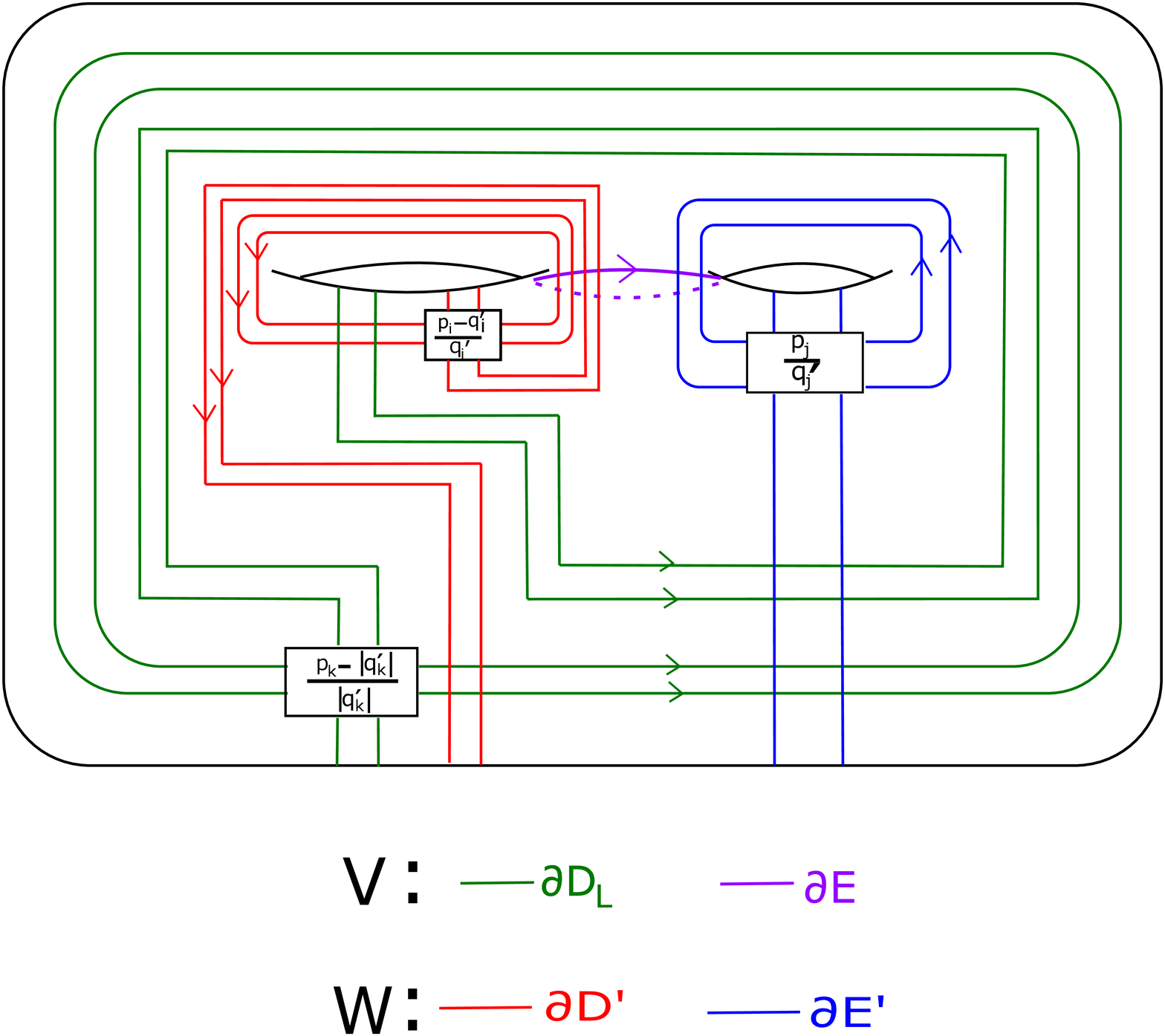}
 \end{center}
 \caption{Useful left standard diagram where $q'_{i}>0$, $0<-q'_{k}<p_k$}
 \label{left_qi+_pk}
\end{figure}

\begin{figure}[htbp]
 \begin{center}
  \includegraphics[width=90mm]{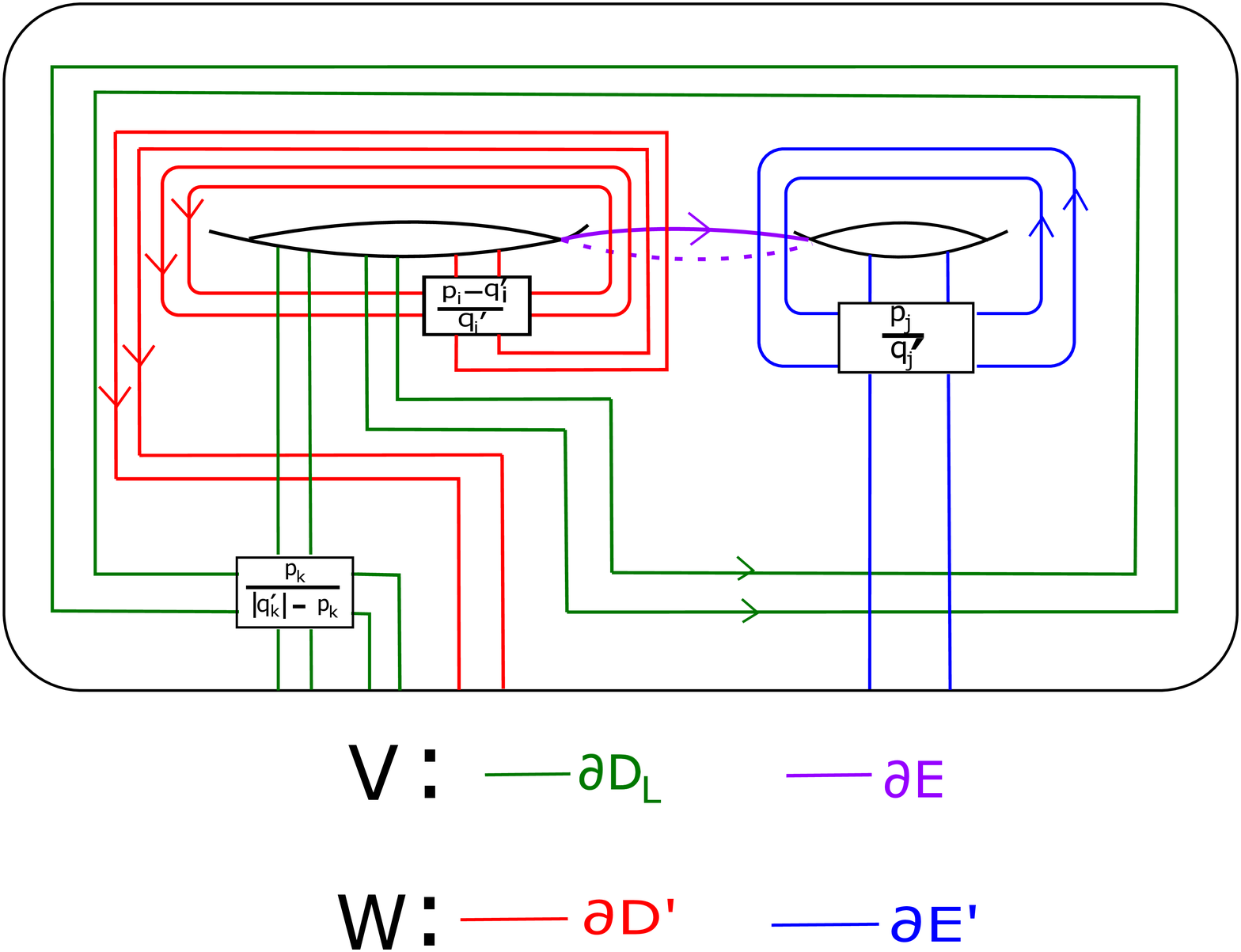}
 \end{center}
 \caption{Useful left standard diagram where $q'_{i}>0$, $0<p_{k}<-q'_{k}$}
 \label{left_qi+_qk}
\end{figure}

\begin{figure}[htbp]
 \begin{center}
  \includegraphics[width=90mm]{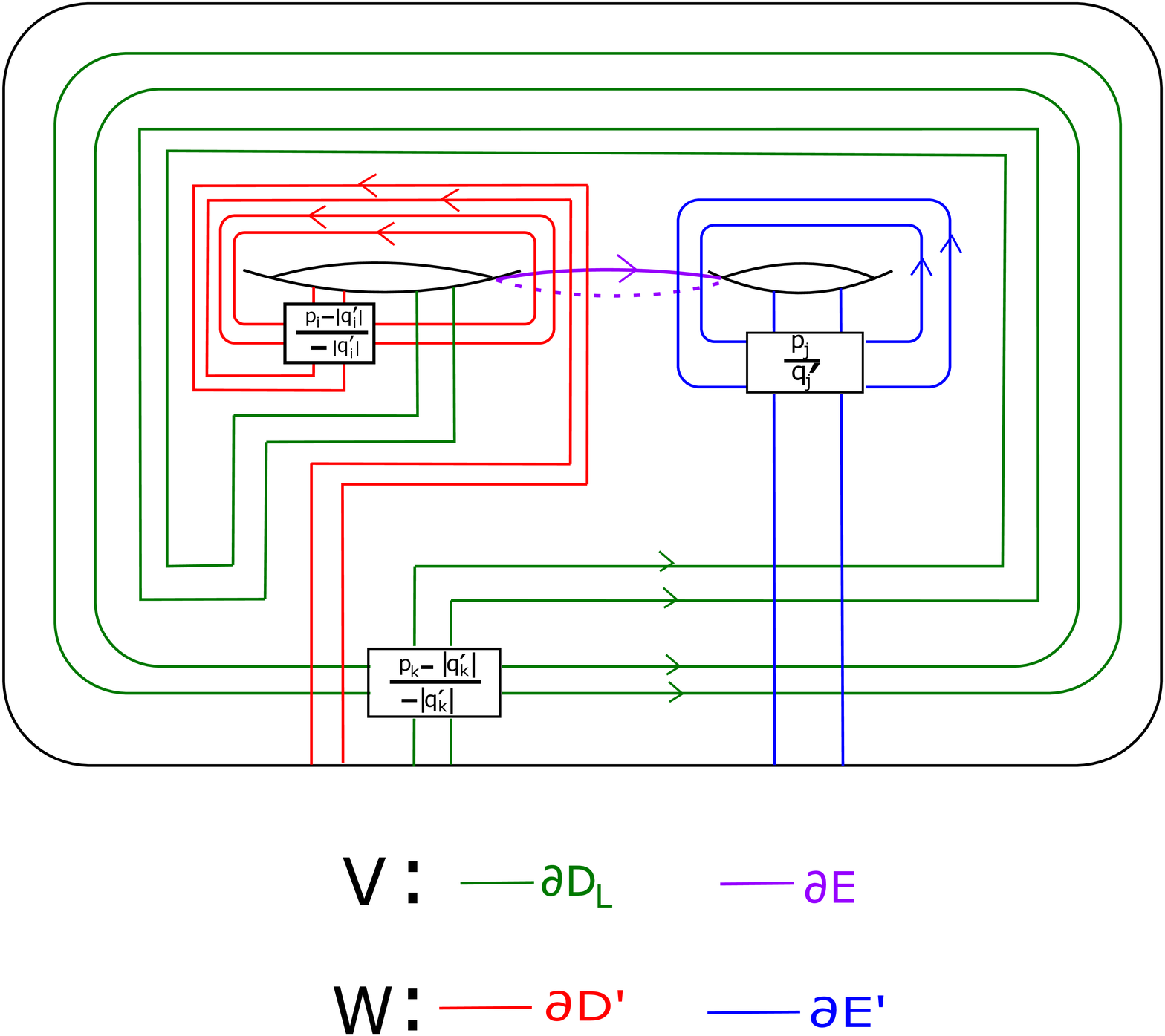}
 \end{center}
 \caption{Useful left standard diagram where $q'_{i}<0$, $0<q'_{k}<p_k$}
 \label{left_qi-_pk}
\end{figure}

\begin{figure}[htbp]
 \begin{center}
  \includegraphics[width=90mm]{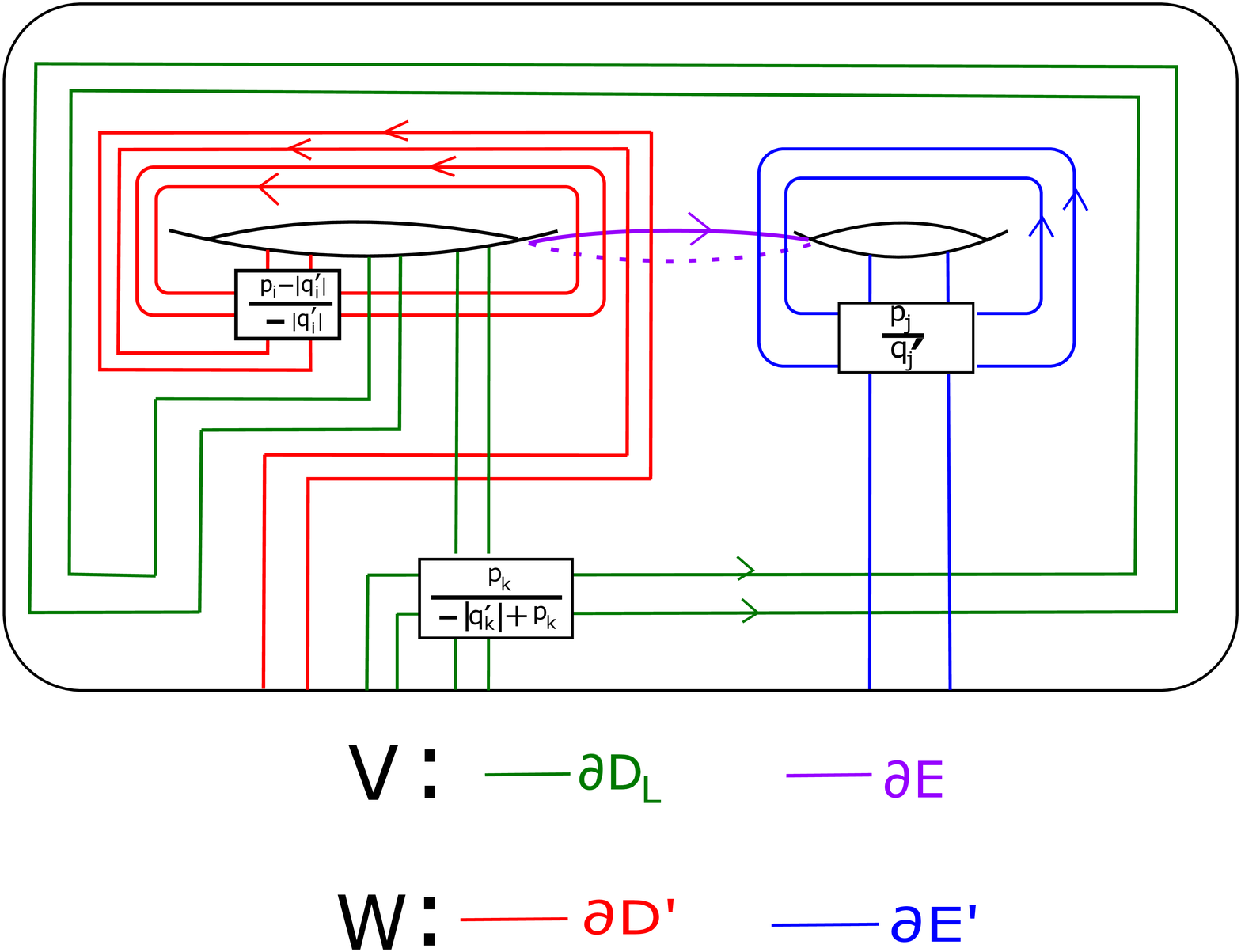}
 \end{center}
 \caption{Useful left standard diagram where $q'_{i}<0$, $0<p_{k}<q'_{k}$}
 \label{left_qi-_qk}
\end{figure}

\begin{figure}[htbp]
 \begin{center}
  \includegraphics[width=90mm]{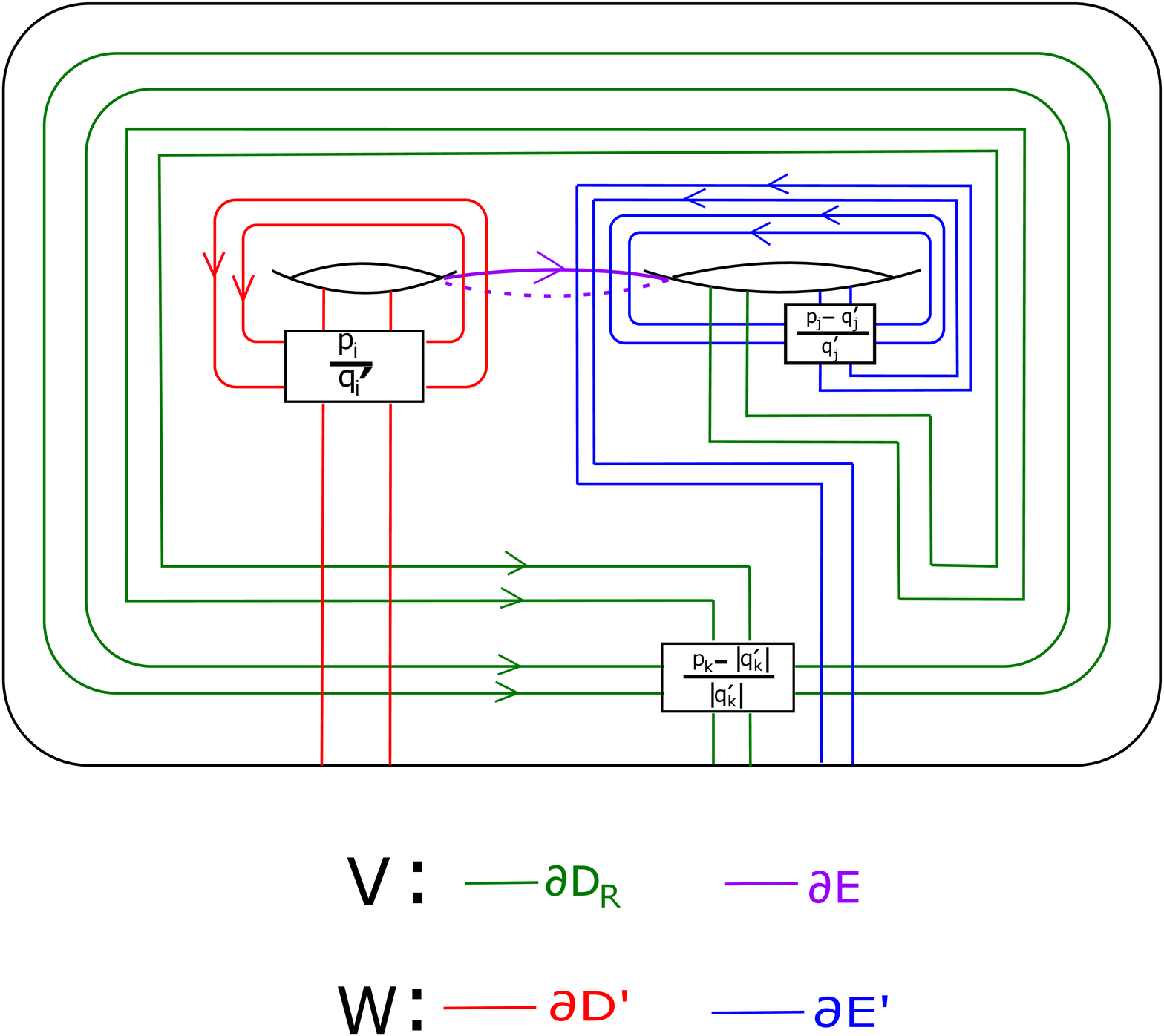}
 \end{center}
 \caption{Useful right standard diagram where $q'_{j}>0$, $0<-q'_{k}<p_k$}
 \label{right_qj+_pk}
\end{figure}

\begin{figure}[htbp]
 \begin{center}
  \includegraphics[width=90mm]{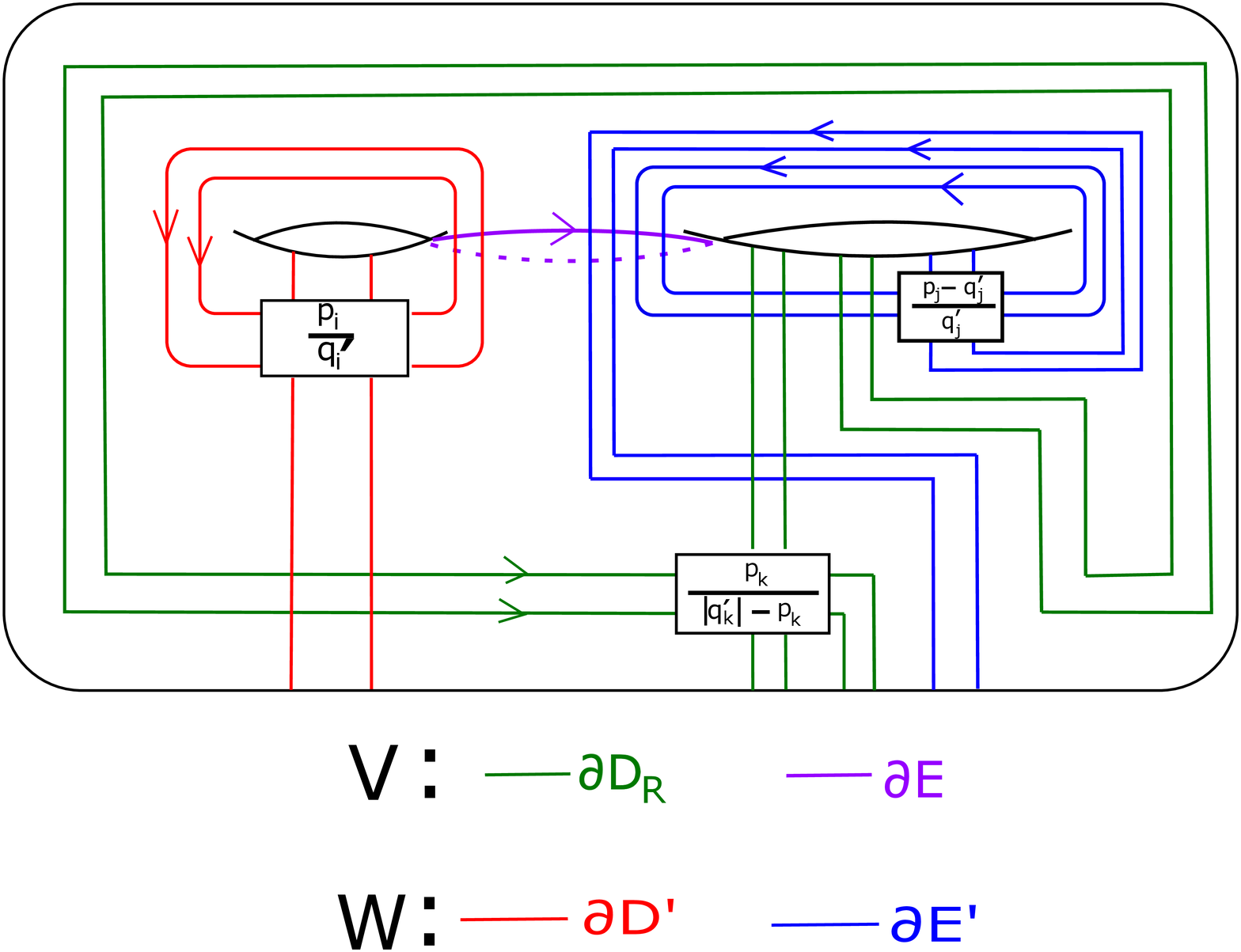}
 \end{center}
 \caption{Useful right standard diagram where $q'_{j}>0$, $0<p_{k}<-q'_{k}$}
 \label{right_qj+_qk}
\end{figure}

\begin{figure}[htbp]
 \begin{center}
  \includegraphics[width=90mm]{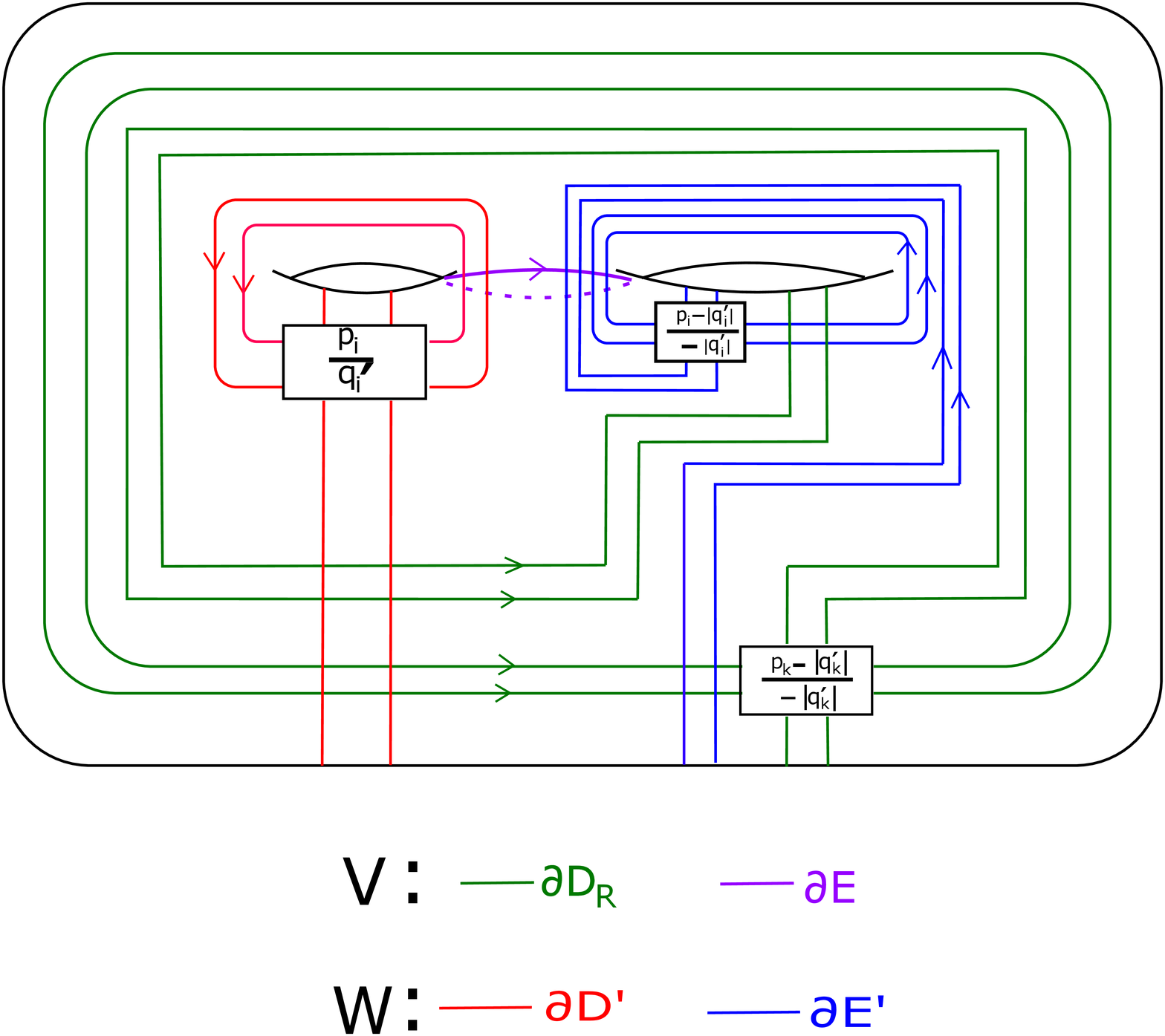}
 \end{center}
 \caption{Useful right standard diagram where $q'_{j}<0$, $0<q'_{k}<p_k$}
 \label{right_qj-_pk}
\end{figure}

\begin{figure}[htbp]
 \begin{center}
  \includegraphics[width=90mm]{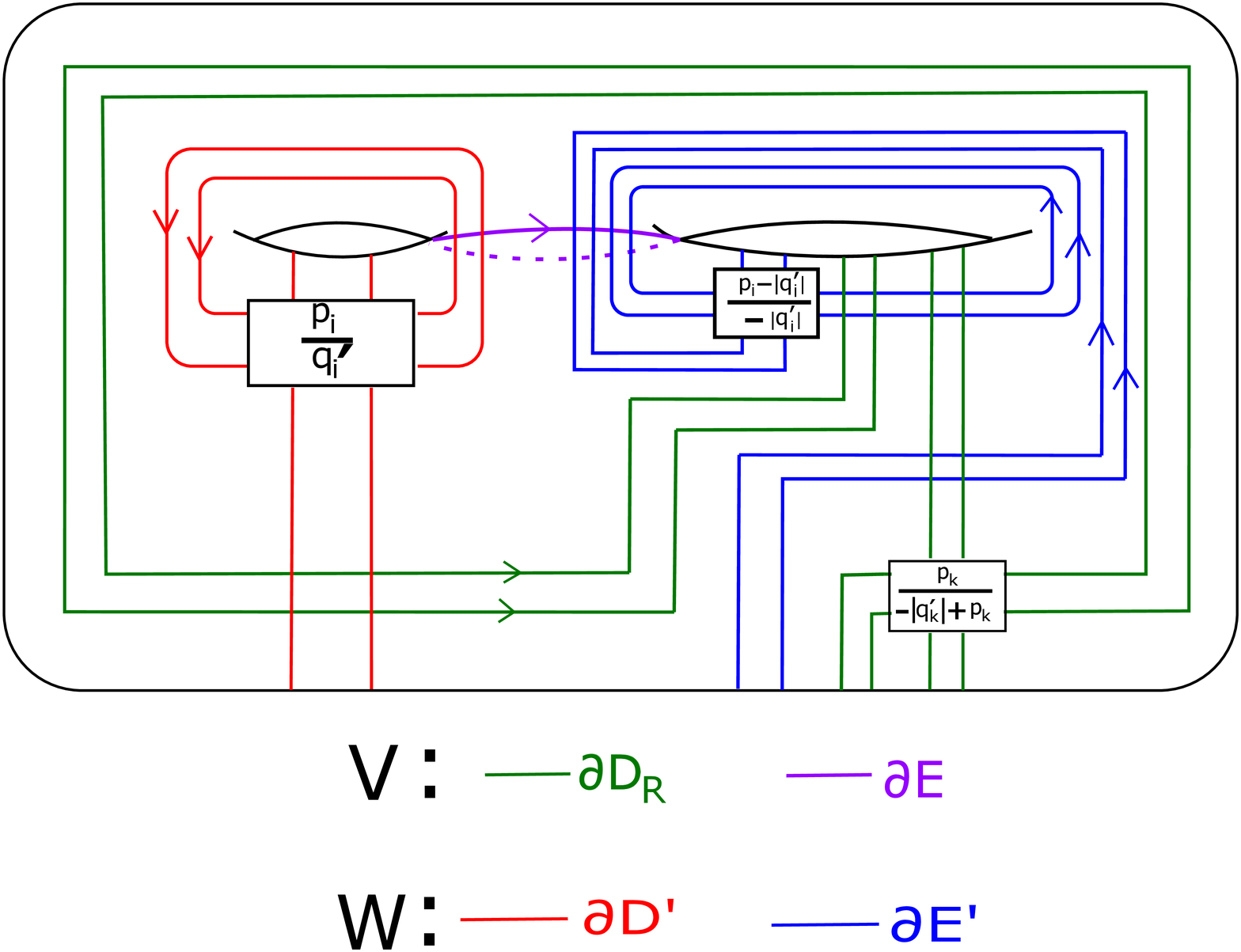}
 \end{center}
 \caption{Useful right standard diagram where $q'_{j}<0$, $0<p_{k}<q'_{k}$}
 \label{right_qj-_qk}
\end{figure}

\section{Terminologies and observations}\label{t_o}
In this section, we review terminologies and observations about simple closed curves on a connected closed oriented surface of genus two. 
Almost of all techniques in this section is in \cite{ck1},\cite{ck2},\cite{ck3},\cite{ck4}. 
Let $S$ be a connected closed oriented surface of genus two in this section.

\begin{defini}(Cut systems)\\
A pair of (oriented) disjoint simple closed curves $\{l_1,l_2\}$ in $S$ is a {\it cut system} of $S$ if $S\setminus(l_1\cup l_2)$ is a connected planar surface. 
Note that curves of a cut system are necessarily non-separating in $S$, and disjoint non-parallel non-separating two simple closed curves in $S$ always form a cut system of $S$. 
\end{defini}

\begin{defini}(Planar diagrams) \label{planar} \\
Let $\{l_1,l_2\}$ be a cut system of $S$. 
When we consider simple closed curves on $S$, it is sometimes convenient to regard them as curves on $S\setminus (l_1\cup l_2)$, denoted by $\mathcal{P}_{\{l_1,l_2\}}$, and we call the obtained diagram {\it planar diagram}. 
The orientation of the interior of $S$ induces that of $\mathcal{P}_{\{l_1,l_2\}}$. 
There are four boundary components in $\mathcal{P}_{\{l_1,l_2\}}$. 
For $i=1,2$, let $l^{+}_{i}$ be the boundary component of $\mathcal{P}_{\{l_1,l_2\}}$ which comes from $l_{i}$ and whose orientation as the boundary of $\mathcal{P}_{\{l_1,l_2\}}$ is same as the orientation induced by $l_{i}$,  and let $l^{-}_{i}$ be the other boundary component coming from $l_{i}$. 
%When $\{D,E\}$ has an orientation, we suppose that the ``right'' side of $\partial D$ (resp. $\partial E$) corresponds to $D^{+}$ (resp. $E^{+}$).
\end{defini}

\begin{defini}(Band sums)\\
Let $\{l_1,l_2\}$ be a cut system of $S$. 
Take an oriented arc $\alpha$ on $S$ starting from $l_1$, ending on $l_2$ and disjoint from $l_1\cup l_2$ in its interior. 
Then the boundary of $N(l_1\cup \alpha \cup l_2)$ which is neither $l_1$ nor $l_2$, denoted by $l'_{1}$ is a non-separating simple closed curve in $S$ which is parallel to neither $l_1$ nor $l_2$. 
We give $l'_1$ the orientation coming from $l_2$, which has parallel orientation near $l_2$.
We call this $l'_{1}$ the curve obtained by a {\it band-sum} from $\{l_1,l_2\}$ (using a {\it band} $\alpha$). 
Note that $\{l'_1,l_2\}$ is another cut system of $S$, and we call this a {\it cut system obtained by a band-sum} from $\{l_1,l_2\}$ (using a band $\alpha$). 
%Since every non separating disk of $V$ which is disjoint from $D\cup E$ is obtained by some band sum from $\{D,E\}$ and the non separating disk complex of a handlebody is connected [???], two cut systems are related by a sequence of band-sums.
Moreover, when $S$ is the boundary of a handlebody $V$ of genus two and $l_1$ and $l_2$ bound disjoint disks $D_1$ and $D_2$ in $V$, a curve obtained by a band-sum from $\{l_1,l_2\}$ also bounds a disk $D'_1$ in $V$ which is disjoint from $D_{1}\cup D_{2}$. 
We also call this $D'_1$ a disk obtained by a band sum from $\{D_1, D_2\}$.
\end{defini}

\begin{defini}(A sequence of letters and  a word)\\
Let $\{l_1,l_2\}$ be a cut system of $S$. 
%Suppose $D$ and $E$ have orientations. 
For an oriented simple closed curve $l$ in $S$, we assign a cyclic sequence of letters $s_{\{l_1,l_2\}}(l)$ as follows: 
Give $l_1$ and $l_2$ letters $x_{l_1}$ and $x_{l_2}$, respectively. 
Follow $l$ along its orientation from arbitrary point on it, and read its intersections with $l_1$ and $l_2$ with signs (, we write the sequence from left to right). 
We decide the signs so that if $l$ is hit by $l_i$ from the right (left) of $l$, then the letter corresponding to this intersection is $x_{l_i}$ ($x^{-1}_{l_i}$, respectively) for $i=1,2$. 
Note that $s_{\{l_1,l_2\}}(l)$ is not reduced as a cyclic word i.e. may have $x_{l_1}x^{-1}_{l_1}$ and so on. 
By reducing $s_{\{l_1,l_2\}}(l)$ as a cyclic word, we get a cyclic word $w_{\{l_1,l_2\}}(l)$. 
If we give $l_i$ the other orientation, then the signs of all $x_{l_i}$ in $s_{\{l_1,l_2\}}(l)$ and $w_{\{l_1,l_2\}}(l)$ are reversed for $i=1,2$. 
Clearly, $s_{\{l_1,l_2\}}(l)=w_{\{l_1,l_2\}}(l)$ is equivalent to that $s_{\{l_1,l_2\}}(l)$ is cyclically reduced. 
We will write cyclic sequences of letters and words as if they were not cyclic because of space limitation, and will omit the term ``cyclic''. 
\end{defini}

\begin{rmk} \label{changeunderbs} (Change of $w_{\{l_1,l_2\}}(l)$ under a band sum)\\
Let $\{l_1,l_2\}$ be a cut system of $S$. 
Suppose $\{l'_1,l_2\}$ is a cut system of $S$ obtained by a band sum from $\{l_1,l_2\}$ using a band $\alpha$, which starts from $l_1$ and ends on $l_2$. 
Let $l$ be an oriented simple closed curve in $S$. 
Then $w_{\{l'_1,l_2\}}(l)$ is computed from $w_{\{l_1,l_2\}}(l)$: 
Isotope $l$ so that its endpoints are disjoint from the endpoints of $\alpha$, and 
regard $l'_1$ as the boundary of (sufficiently small) $N(l_1\cup \alpha \cup l_2)$ which is neither $l_1$ nor $l_2$. 
Since as a word (not a sequence of letters) inessential intersections of $l$ and $l'_1$ are canceled, we can compute $w_{\{l'_1,l_2\}}(l)$ with this representative. 
The intersections of $l$ and $l'_1$ is decomposed into three: Coming from that of $l$ and $\alpha$, coming from that of $l$ and $l_1$, and coming from that of $l$ and $l_2$. 
The intersections coming from $l$ and $\alpha$ is canceled as a word. 
In the other cases, results depends on the sides on which the endpoints of $\alpha$ (as an arc on a planar diagram with respect to $\{l_1,l_2\}$): \\
(1) When the starting point of $\alpha$ is on $l^{+}_{1}$ and terminal point of $\alpha$ is on $l^{+}_{2}$ (or when the starting point of $\alpha$ is on $l^{-}_{1}$ and terminal point of $\alpha$ is on $l^{+}_{2}$), every intersection corresponds to $x_{l_1}$ with respect to $\{x_{l_1},x_{l_2}\}$ becomes $x_{l'_{1}}$ with respect to $\{x_{l'_{1}},x_{l_2}\}$ and that corresponds to $x_{l_2}$ with respect to $\{x_{l_1},x_{l_2}\}$ becomes $x_{l_2}x_{l'_{1}}$ with respect to $\{x_{l'_{1}},x_{l_2}\}$. 
Thus $w_{\{l'_{1},l_2\}}(l)$ is obtained from $w_{\{l_1,l_2\}}(l)$ by replacing $x_{l_1}$ (and $x^{-1}_{l_1}$) to $x_{l'_{1}}$ (and $x^{-1}_{l'_{1}}$) and $x_{l_2}$ (and $x^{-1}_{l_2}$) to $x_{l_2}x_{l'_{1}}$ (and $x^{-1}_{l'_{1}}x^{-1}_{l_2}$).\\
(2) When the starting point of $\alpha$ is on $l^{+}_{1}$ and terminal point of $\alpha$ is on $l^{-}_{2}$ (or when the starting point of $\alpha$ is on $l^{-}_{1}$ and terminal point of $\alpha$ is on $l^{+}_{2}$), every intersection corresponds to $x_{l_1}$ with respect to $\{x_{l_1},x_{l_2}\}$ becomes $x^{-1}_{l'_{1}}$ with respect to $\{x_{l'_{1}},x_{l_2}\}$ and that corresponds to $x_{l_2}$ with respect to $\{x_{l_1},x_{l_2}\}$ becomes $x_{l'_{1}}x_{l_2}$ with respect to $\{x_{l'_{1}},x_{l_2}\}$. 
Thus $w_{\{l'_1,l_2\}}(l)$ is obtained from $w_{\{l_1,l_2\}}(l)$ by replacing $x_{l_1}$ (and $x^{-1}_{l_1}$) to $x^{-1}_{l'_1}$ (and $x_{l'_{1}}$) and $x_{l_2}$ (and $x^{-1}_{l_2}$) to $x_{l'_{1}}x_{l_2}$ (and $x^{-1}_{l_2}x^{-1}_{l'_{1}}$).
\end{rmk}

\begin{defini}(A corresponding arc)\\
Let $\{l_1,l_2\}$ be a cut system of $S$ and $l$ an oriented simple closed curve in $S$. 
Suppose $s_{\{l_1,l_2\}}(l)$ has one of $x_{l_1}x^{-1}_{l_1}$,$x^{-1}_{l_1}x_{l_1}$,$x_{l_2}x^{-1}_{l_2}$ and $x^{-1}_{l_2}x_{l_2}$, say $x_{l_1}x^{-1}_{l_1}$. 
This subsequence of letters corresponds to a subarc $c$ of $l$ starting from and ending on $l^{+}_{1}$ in the planar diagram $\mathcal{P}_{\{l_1,l_2\}}$. 
This $c$ separates other three boundary components of $\mathcal{P}_{\{l_1,l_2\}}$ into two sets, none of which are empty since $l$ intersects with $l_1$ essentially. 
Thus there is an oriented arc $\alpha$ on $\mathcal{P}_{\{l_1,l_2\}}$ starting from $l^{+}_{1}$ and ending on another boundary component, denoted by $A$, such that $c$ is the interior of the boundary of $N\left( \alpha \cup A;{\rm int}(\mathcal{P}_{\{l_1,l_2\}})\right)$. 
This $\alpha$ determines $c$ except for the orientation, and is called the {\it corresponding arc} of $c$. 
This $A$ is not $l^{-}_{1}$ since otherwise $|l \cap l^{+}_{1}|$ were grater than $|l \cap l^{-}_{1}|$. 
Note that if there is another subarc $c'$ of $l$ representing $x_{l_1}x^{-1}_{l_1}$, then its corresponding arc $\alpha'$ is isotopic (endpoints can move on cut ends) to $\alpha$. 
This is because $c$ and $c'$ would intersect otherwise.
\end{defini}

\begin{obs} \label{cc-1}
Let $\{l_1,l_2\}$ be a cut system of $S$ and $l$ an (oriented) simple closed curve on $S$. 
If there are $n(\geq0)$ subarcs of $l$ starting from and ending on $l^{+}_1$ in $\mathcal{P}_{\{l_1,l_2\}}$ (corresponding to $x_{l_1}x^{-1}_{l_1}$), then there are also $n$ subarcs of $l$ starting from and ending on $l^{-}_1$ (corresponding to $x^{-1}_{l_1}x_{l_1}$).
\end{obs}
\begin{proof}
We assume that the corresponding arc of (one of) subarcs of $l$ starting from and ending on $l^{+}_1$ ending on $l^{+}_2$. 
In $\mathcal{P}_{\{l_1,l_2\}}$, $l$ is cut into disjoint arcs since $l$ intersects $l_1$. 
Suppose that there are exactly $n(>0)$ subarcs of $l$ starting from and ending on $l^{+}_1$ in $\mathcal{P}_{\{l_1,l_2\}}$. We want to show that there are $n$ subarcs of $l$ starting from and ending on $l^{-}_1$ in $\mathcal{P}_{\{l_1,l_2\}}$.\\
At first, suppose that there are exactly $m>0$ subarcs of $l$ starting from and ending on $l^{-}_2$ in $\mathcal{P}_{\{l_1,l_2\}}$. 
Then the other subarcs are one of three types: Arcs connecting $l^{+}_1$ and $l^{+}_{2}$, arcs connecting $l^{-}_2$ and $l^{-}_{1}$, and arcs connecting $l^{-}_2$ and $l^{+}_{1}$. 
Let $a,b,c$ be the non-negative numbers of the arcs of the three types, respectively. 
Then we have $2n+a+c=b$ for $|l\cap l_1|$ and $2m+b+c=a$ for $|l\cap l_2|$. By combining these, we have $n+m+c=0$, this contradicts to $n,m>0$. \\
Thus there are no subarcs of $l$ starting from and ending on $l^{-}_2$ in $\mathcal{P}_{\{l_1,l_2\}}$. 
Next, suppose that there are no subarcs of $l$ starting from and ending on $l^{-}_1$ in $\mathcal{P}_{\{l_1,l_2\}}$. 
Then the other subarcs are one of for types: Arcs connecting $l^{+}_1$ and $l^{+}_{2}$, arcs connecting $l^{-}_2$ and $l^{-}_{1}$, arcs connecting $l^{-}_2$ and $l^{+}_{1}$, and arcs connecting $l^{-}_1$ and $l^{+}_{1}$. 
Let $a,b,c,d$ be the non-negative numbers of the arcs of the three types, respectively. 
Then we have $2n+a+c+d=b+d$ for $|l\cap l_1|$ and $b+c=a$ for $|l\cap l_2|$. By combining these, we have $n+c=0$, this contradicts to $n>0$. \\
Thus there is at least one subarc of $l$ starting from and ending on $l^{-}_1$ in $\mathcal{P}_{\{l_1,l_2\}}$. Let $m>0$ be the number of such subarcs. 
Then the other subarcs are one of three types: Arcs connecting $l^{+}_1$ and $l^{+}_{2}$, arcs connecting $l^{-}_2$ and $l^{-}_{1}$, and arcs connecting $l^{-}_1$ and $l^{+}_{1}$. 
Let $a,b,c$ be the non-negative numbers of the arcs of the three types, respectively. 
Then we have $2n+a=2m+b$ for $|l\cap l_1|$ and $a=b$ for $|l\cap l_2|$. By combining these, we have $n=m$, this completes the proof. 
\end{proof}

\begin{obs}\label{noncoexist}
Let $\{l_1,l_2\}$ be a cut system of $S$ and $l$ an oriented simple closed curve on $S$. 
If there is a subsequence of letters $x_{l_1}x^{-1}_{l_1}$ in $s_{\{l_1,l_2\}}(l)$, then there are no subsequences of letters $x_{l_2}x^{-1}_{l_2}$ in $s_{\{l_1,l_2\}}(l)$. 
Moreover, if there is another oriented simple closed curve $l'$ on $S$ disjoint from $l$, then there are no subsequences of letters $x_{l_2}x^{-1}_{l_2}$ in $s_{\{l_1,l_2\}}(l')$, neither.
\end{obs}

\begin{proof}
Take a subarc $c$ of $l$ in $\mathcal{P}_{\{l_1,l_2\}}$ representing $x_{l_1}x^{-1}_{l_1}$ in $s_{\{l_1,l_2\}}(l)$. 
We assume that the corresponding arc of $c$ connects $l^{+}_{1}$ and $l^{+}_2$. 
Suppose for contradiction that $l$ (or $l'$, disjoint from $l$) has a subarc representing $x_{l_2}x^{-1}_{l_2}$ in $s_{\{l_1,l_2\}}(l)$ (or in $s_{\{l_1,l_2\}}(l')$). 
Then by Observation \ref{cc-1}, $l$ (or $l'$) has a subarc in $\mathcal{P}_{\{l_1,l_2\}}$ both of whose endpoints are on $l^{+}_{2}$. 
This subarc would intersect $c$. 
\end{proof}

\begin{obs} \label{parallel}
Let $\{l_1,l_2\}$ be a cut system of $S$, $c$ be an oriented arc in $\mathcal{P}_{\{l_1,l_2\}}$ starting from and ending on $l^{+}_1$ (or starting from and ending on $l^{-}_{1}$) or 
a curve obtained by a band sum operation from $\{l_1,l_2\}$ using some arc. 
Suppose there are consecutive simple (unoriented) parallel arcs $a_{1}, \dots, a_{n}$ connecting $l^{+}_2$ and $l^{-}_2$ in $\mathcal{P}_{\{l_1,l_2\}}$ such that the set of the endpoints on $l^{+}_{2}$ are ordered counterclockwise as the numbering. 
Then a part of $c$ or $-c$ (, which is $c$ with the other orientation) hits consecutively $a_{1},\dots,a_{n}$. 
Moreover if $a_{1},\dots,a_{n}$ have a parallel orientation, then the signs of the intersections of $a_{1},\dots,a_{n}$ and the above part of $c$ or $-c$ is same.
\end{obs}

\begin{proof}
Take the corresponding arc (or the band) of $c$. We call this arc (or band) $\alpha$. 
We assume that the endpoint of $\alpha$ on $l_2$ in $S$ is on $l^{+}_{2}$ in $\mathcal{P}_{\{l_1,l_2\}}$. The case where it is on $l^{-}_{2}$ in $\mathcal{P}_{\{l_1,l_2\}}$ is similar. 
Then the statements hold by noting that $c$ is isotopic to the boundary of the neighborhood of $\alpha$ and one of (or both of) boundaries which have endpoints of $\alpha$. 
See Figure \ref{parallels}.
\end{proof}

\begin{figure}[htbp]
 \begin{center}
  \includegraphics[width=60mm]{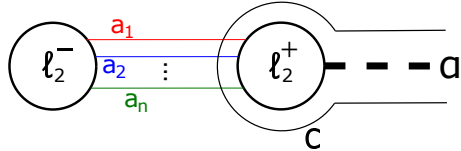}
 \end{center}
 \caption{$n$ parallel arcs and $c$}
 \label{parallels}
\end{figure}

%\begin{obs}\label{ck1} [ck1]
%Let $\{D,E\}$ be a cut system of $V$, $l$ an oriented simple closed curve on $\partial V$. 
%If $s_{\{D,E\}}(l)$ has $x_{D}x^{n}_{E}x^{-1}_{D}$ for non zero integer $n$, then $s_{\{D,E\}}(l)=w_{\{D,E\}}(l)$.
%\end{obs}

\begin{obs}\label{ck2} \cite{ck2}
Let $\{l_1,l_2\}$ be a cut system of $S$, $l$ an oriented simple closed curve on $S$. 
Suppose $s_{\{l_1,l_2\}}(l)$ has one of $x^{2}_{l_1}$ and $x^{-2}_{l_1}$, and also has one of $x^{2}_{l_2}$ and $x^{-2}_{l_2}$. 
Then $s_{\{l_1,l_2\}}(l)=w_{\{l_1,l_2\}}(l)$.
\end{obs}

\begin{proof}
Since $s_{\{l_1,l_2\}}(l)$ has one of $x^{2}_{l_1}$ and $x^{-2}_{l_1}$, and also has one of $x^{2}_{l_2}$ and $x^{-2}_{l_2}$, the planar diagram of $l$ on $\mathcal{P}_{\{l_1,l_2\}}$ has an arc connecting $l^{+}_{1}$ and $l^{-}_{1}$, and an arc connecting $l^{+}_{2}$ and $l^{-}_{2}$. 
Then there are no arc of $l$ in $\mathcal{P}_{\{l_1,l_2\}}$ whose endpoints are on one boundary component, otherwise $l$ has self intersection by Observation \ref{parallel}. 
This means $s_{\{l_1,l_2\}}(l)=w_{\{l_1,l_2\}}(l)$.
\end{proof}

Of course, the above definitions, remark and observations also holds if we change the roles of $l_1$ and $l_2$. 

\section{A computation}\label{computation}
We fix a Heegaard splitting $M=V\cup_{S}W$ as in Subsection \ref{hs} and we work on the useful standard diagrams after being minimally intersect. 
In \cite{cho1},\cite{cho2},\cite{ck1},\cite{ck2},\cite{ck3},\cite{ck4}, Cho and Cho, Koda gave finite presentations of Goeritz groups of reducible genus two Heegaard splittings by using their actions on some subcomplexes of the curve complexes of splitting surfaces, which has enough information about the splittings. 
%In a part of [ik], IK proved that the Goeritz groups of Heegaard splittings which have some finiteness about the Hempel distance are finite group. 
We will follow their techniques. 
We use the following subcomplex of the curve complex. In fact, we need not the structure as a complex. Hence we define only as set.

\begin{defini}
$\mathcal{C}$ denotes the set of the isotopy classes of every essential disk $D$ in $V$ such that there exist two disks $\tilde{D'},\tilde{E'}$ in $W$ such that $\{\partial \tilde{D'},\partial \tilde{E'}\}$ is a cut system of $S(=\partial W)$ and $s_{\{\partial \tilde{D}', \partial \tilde{E}'\}}(\partial D)=x^{p_i}_{\partial \tilde{D}'}x^{p_j}_{\partial \tilde{E}'}$ under appropriate orientations of $D$, $\tilde{D}'$ and $\tilde{E}'$.
\end{defini}

Clearly, the Goeritz group acts on $\mathcal{C}$. 
The disk $E$ in Figure \ref{std_diag} is an element of $\mathcal{C}$ by setting $\{\tilde{D'},\tilde{E'}\}=\{D',E'\}$. 
The following claim, we postpone the proof, implies that the structure of $M=V\cup_{S}W$ is very rigid: 

\begin{claim}\label{claim1}
$\mathcal{C}=\{E\}$, and the disks $\tilde{D'},\tilde{E'}$ in $W$ in the condition is $(\tilde{D'}, \tilde{E'})=(D', E')$ or $(E', D')$.
\end{claim}

Thus every element of the Goeritz group maps $E$ to $E$ (may reverse the orientation),
 and maps $D'$ and $E'$ to themselves (may reverse the orientations) or exchanges $D'$ and $E'$. 
Note that the necessary condition for happening the exchange is $p_{i}=p_{j}$. 
We need two more claims, we postpone the proofs.

\begin{claim}\label{claim2}
There is an element of the Goeritz group $\mathcal{G}(H_{\{i,j\}})$ exchanges $D'$ and $E'$ if and only if $p_{i}=p_{j}$ and $q_{i}\equiv q_{j}$ {\rm mod} $p_{i}$.
\end{claim}

\begin{claim}\label{claim3}
Every element of the Goeritz group $\mathcal{G}(H_{\{i,j\}})$ mapping $E$ to $E$ preserving the orientation and mapping $D'$ and $E'$ to themselves is the identity element.
\end{claim}

Recall that $h$ maps $E$, $D'$ and $E'$ to themselves reversing the orientations, and that $\iota$ (, if it is defined,) maps $E$ to $E$ preserving the orientation and exchanges $D'$ and $E'$. 
By the above claims, we see that the Goeritz group is generated by $h$ and $\iota$ (, if it is defined). 
Moreover, we can see that $h^{2}$, $\iota^{2}$ and $h\iota h\iota$ satisfy the condition in Claim \ref{claim3}, thus they are identity elements. 
This complete the proof of Theorem \ref{main}.

\subsection{A proof of Claim \ref{claim1}}
Take arbitrary $D \in \mathcal{C}$ and fix disks $\tilde{D'},\tilde{E'}$ in $W$ for $D$ in the definition of $\mathcal{C}$. 
We will divide into the cases.\\
(1)The case where $\tilde{D'}$ intersects with $D'\cup E'$. \\
In this case, choose an outermost disk of $\tilde{D'}$ in $\tilde{D'}\setminus (D'\cup E')$. \\

(1.1)The case where $D'$ cuts the outermost disk of $\tilde{D'}$.\\
Consider the diagram of $\partial D_{L}\cup \partial E$ on $\mathcal{P}_{\{\partial D',\partial E'\}}$. 
Note that there are at least $2$-parallel consecutive arcs with same orientations coming from $\partial D_{L}$ connecting $\partial E'^{+}$ and $\partial E'^{-}$ in $\mathcal{P}_{\{\partial D',\partial E'\}}$, and that there are at least $2$-parallel consecutive arcs with same orientations coming from $\partial E$ connecting $\partial E'^{+}$ and $\partial E'^{-}$ in $\mathcal{P}_{\{\partial D',\partial E'\}}$. 
A part of $\tilde{D'}$, corresponding to the outermost disk, is an arc on $\mathcal{P}_{\{\partial D',\partial E'\}}$ starting and ending on $\partial D'^{+}$ (or on $\partial D'^{-}$). 
Thus $s_{\{\partial D_{L},\partial E\}}(\partial \tilde{D'})$ has both of ``$x_{\partial D_{L}}^{2}$ or $x_{\partial D_{L}}^{-2}$'' and ``$x_{\partial E}^{2}$ or $x_{\partial E}^{-2}$'' by Observation \ref{parallel}. 
For $\tilde{E'}$, there are three possibilities: ``$\tilde{E'}$ intersects with $D'\cup E'$'', ``$\tilde{E'}$ is disjoint from $D'\cup E'$ and neither $D'$ nor $E'$'', and ``$\tilde{E'}=E'$''. 
If $\tilde{E'}$ intersects with $D'\cup E'$, every outermost disk of $\tilde{E'}$ in $\tilde{E'}\setminus (D'\cup E')$ is cut by $D'$ by Observation \ref{noncoexist}. 
Thus $s_{\{\partial D_{L},\partial E\}}(\partial \tilde{E'})$ has both of ``$x_{\partial D_{L}}^{2}$ or $x_{\partial D_{L}}^{-2}$'' and ``$x_{\partial E}^{2}$ or $x_{\partial E}^{-2}$'' by Observation \ref{parallel}. 
If $\tilde{E'}$ is disjoint from $D'\cup E'$ and neither $D'$ nor $E'$, then $s_{\{\partial D_{L},\partial E\}}(\partial \tilde{E'})$ has both of ``$x_{\partial D_{L}}^{2}$ or $x_{\partial D_{L}}^{-2}$'' and ``$x_{\partial E}^{2}$ or $x_{\partial E}^{-2}$'' by Observation \ref{parallel}. 
If $\tilde{E'}=E'$, then $s_{\{\partial D_{L},\partial E\}}(\partial \tilde{E'})$ has both of ``$x_{\partial D_{L}}^{2}$ or $x_{\partial D_{L}}^{-2}$'' and ``$x_{\partial E}^{2}$ or $x_{\partial E}^{-2}$'' by looking the useful standard left diagram. 
Therefore we see that in any cases, $s_{\{\partial D_{L},\partial E\}}(\partial \tilde{E'})$ has both of ``$x_{\partial D_{L}}^{2}$ or $x_{\partial D_{L}}^{-2}$'' and ``$x_{\partial E}^{2}$ or $x_{\partial E}^{-2}$''. 
Next consider the diagram of $\partial \tilde{D'}\cup \partial \tilde{E'}$ on $\mathcal{P}_{\{\partial D_{L},\partial E\}}$. 
Since $s_{\{\partial D_{L},\partial E\}}(\partial \tilde{D'})$ has both of ``$x_{\partial D_{L}}^{2}$ or $x_{\partial D_{L}}^{-2}$'' and ``$x_{\partial E}^{2}$ or $x_{\partial E}^{-2}$'', there are two subarcs of $\partial \tilde{D'}$ in $\mathcal{P}_{\{\partial D_{L},\partial E\}}$, one connects $\partial D^{+}_{L}$ and $\partial D^{-}_{L}$, the other connects $\partial E^{+}$ and $\partial E^{-}$. 
And since $s_{\{\partial D_{L},\partial E\}}(\partial \tilde{E'})$ has both of ``$x_{\partial D_{L}}^{2}$ or $x_{\partial D_{L}}^{-2}$'' and ``$x_{\partial E}^{2}$ or $x_{\partial E}^{-2}$'', there are two subarcs of $\partial \tilde{E'}$ in $\mathcal{P}_{\{\partial D_{L},\partial E\}}$, one connects $\partial D^{+}_{L}$ and $\partial D^{-}_{L}$, the other connects $\partial E^{+}$ and $\partial E^{-}$. 
We call these $4$ arcs {\it the four arcs}. See Figure \ref{p_dle1}. 
Note that two of the four arcs, connecting $\partial D^{+}_{L}$ and $\partial D^{-}_{L}$ are parallel, and the others are also parallel (not necessarily consecutive in the diagram).

\begin{figure}[htbp]
 \begin{center}
  \includegraphics[width=50mm]{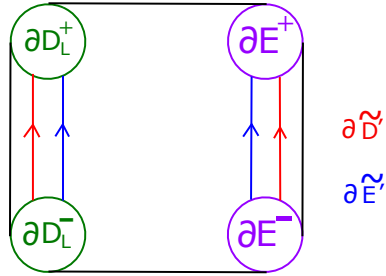}
 \end{center}
 \caption{A part of the diagram of $\partial \tilde{D'}\cup \partial \tilde{E'}$ on $\mathcal{P}_{\{\partial D_{L},\partial E\}}$}
 \label{p_dle1}
\end{figure}

(1.1.1)The case where $D$ intersects with $D_{L}\cup E$.\\
In this case, choose a outermost disk of $D$ in $D\setminus (D_{L}\cup E)$. \\

(1.1.1.1)The case where $D_{L}$ cuts the outermost disk of $D$. \\
Let $c$ be a subarc of $\partial D$ corresponding to the outermost disk. 
This $c$ is a arc on $\mathcal{P}_{\{\partial D_{L},\partial E\}}$ starting and ending on $\partial D^{+}_{L}$ or $\partial D^{-}_{L}$, say $\partial D^{+}_{L}$. 
By Observation \ref{cc-1}, there is a subarc $c'$ of $\partial D$ starting and ending on $\partial D^{-}_{L}$ on $\mathcal{P}_{\{\partial D_{L},\partial E\}}$. 
The corresponding arcs of $c$ and $c'$ are disjoint. 
Moreover, these corresponding arcs are disjoint from the four arcs since the signs of the intersection points of $\partial D$ and $\partial \tilde{D'}$ are same. 
By the signs of the intersection mentioned above, the orientation of $c\cup c'$ has two possibilities. 
In any cases, there must be a pair of two points in $\partial D\cap \partial \tilde{D'}$ and a pair of two points in $\partial D\cap \partial \tilde{E'}$ such that the former pair is separated by the latter pair on $\partial D$. See Figure \ref{ccdisk}. 
This contradicts to the property of $D$ and $\{\tilde{D'},\tilde{E'}\}$. 

\begin{figure}[htbp]
 \begin{center}
  \includegraphics[width=100mm]{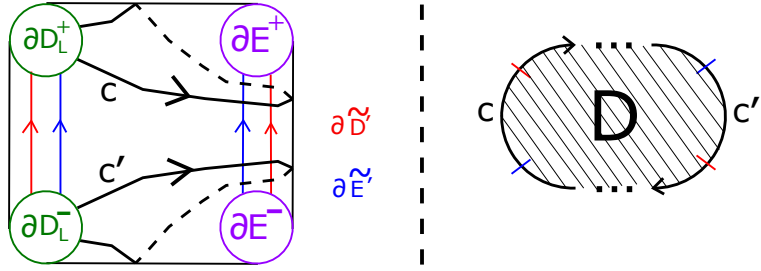}
 \end{center}
 \caption{An example of $c\cup c'$ }
 \label{ccdisk}
\end{figure}

(1.1.1.2)The case where $E$ cuts the outermost disk of $D$.\\
In this case, the argument almost similar to that in (1.1.1.1) leads a contradiction.\\

(1.1.2)The case where $D$ and $D_{L}\cup E$ are disjoint.\\
Note that in the diagram of $\partial D' \cup \partial E'$ in $\mathcal{P}_{\{\partial D_{L},\partial E\}}$,
 there are $(p_{k}-1)$ blocks of $|q'_{j}|(\geq 2)$ consecutive parallel oriented arcs coming from $\partial E'$ connecting $\partial D^{+}_{L}$ and $\partial D^{-}_{L}$, with the parallel orientation. 
Note that in this diagram, there are at least two consecutive parallel arcs connecting $\partial D^{+}_{L}$ and $\partial D^{-}_{L}$ coming from $\partial D'$ with the parallel orientation. 
Note that in this diagram, there are $p_{j}-|q'_{j}|(\geq 3)$ consecutive parallel arcs connecting $\partial E^{+}$ and $\partial E^{-}$ coming from $\partial E'$ with the parallel orientation. 
Note that $E$ intersects with $\partial D'$ more than twice in identical signs. 
See Figure \ref{p_dle2}. 

\begin{figure}[htbp]
 \begin{center}
  \includegraphics[width=50mm]{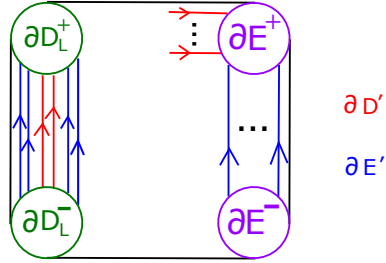}
 \end{center}
 \caption{A part of the diagram of $\partial D'\cup \partial E'$ on $\mathcal{P}_{\{\partial D_{L},\partial E\}}$}
 \label{p_dle2}
\end{figure}

By the above, in the diagram of $\partial D$ in $\mathcal{P}_{\{\partial D',\partial E'\}}$, there are an arc connecting $\partial D'^{+}$ and $\partial D'^{-}$, and two parallel arcs (not necessarily consecutive) connecting $\partial E'^{+}$ and $\partial E'^{-}$ with the same orientations, we say these three arcs {\it the three arcs}. 
See Figure \ref{p_dwew1}.

\begin{figure}[htbp]
 \begin{center}
  \includegraphics[width=40mm]{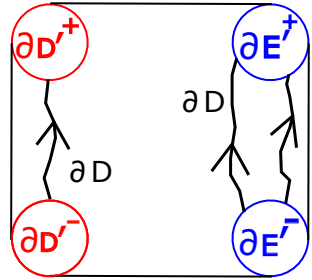}
 \end{center}
 \caption{A part of the diagram of $\partial D$ on $\mathcal{P}_{\{\partial D',\partial E'\}}$}
 \label{p_dwew1}
\end{figure}

Now we are in (1.1), there is a subarc $d$ of $\partial \tilde{D}'$ starting and ending on $\partial D'^{+}$ (or $\partial D'^{-}$), say $\partial D'^{+}$. 
By Observation \ref{cc-1}, there is also a subarc $d'$ of $\partial \tilde{D}'$ starting and ending on $\partial D'^{-}$. 
The corresponding arcs of $d$ and $d'$ are disjoint. 
Moreover, the corresponding arcs of $d$ and $d'$ are disjoint from the three arcs since the signs of the intersection points of $\partial D$ and $\partial \tilde{D'}$ are same. 
By the signs of the intersection mentioned above, the orientation of $d\cup d'$ has two possibilities. 
For $\tilde{E'}$, there are three possibilities: ``$\tilde{E'}$ intersects with $D'\cup E'$ (, and every outermost disk of $\tilde{E'}$ is cut by $D'$ by Observation \ref{noncoexist})'', ``$\tilde{E'}$ is disjoint from $D'\cup E'$ and neither $D'$ nor $E'$'', ``$\tilde{E'}=E'$''. 
In any cases, there must be a pair of two points in $D\cap \tilde{D'}$ and a pair of two points in $D\cap \tilde{E'}$ such that the former pair is separated by the latter pair on $\partial D$, see Figure \ref{possi_three}. 
This contradicts to the property of $D$ and $\{\tilde{D'},\tilde{E'}\}$.  

\begin{figure}[htbp]
 \begin{center}
  \includegraphics[width=120mm]{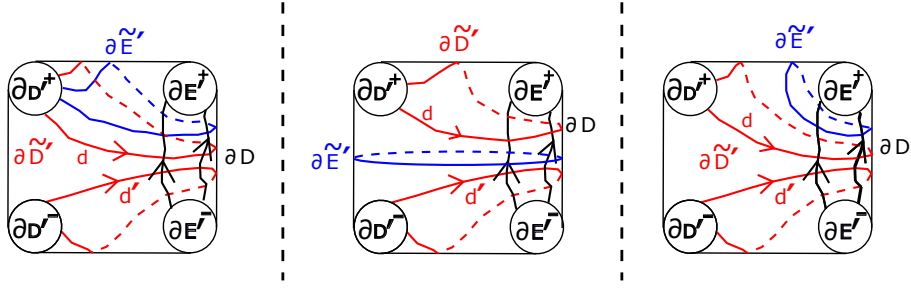}
 \end{center}
 \caption{Three possibilities for $\tilde{E'}$}
 \label{possi_three}
\end{figure}

(1.2)The case where $E'$ cuts the outermost disk of $\tilde{D'}$.\\
By replacing the useful left standard diagram with the useful right standard diagram in the argument in (1.1), we get a contradiction.\\

(2)The case where $\tilde{E'}$ intersects with $D'\cup E'$. \\
By the same argument in (1), we get a contradiction.\\

(3)The case where $\tilde{D'}$ and $\tilde{E'}$ are disjoint from $D'\cup E'$.\\
In this case, at least one of $\tilde{D'}$ and $\tilde{E'}$ is $D'$ or $E'$. \\

(3.1)The case where $\tilde{E'}$ is neither $D'$ nor $E'$.\\
Then $\tilde{D'}$ is $D'$ or $E'$, say $D'$, and $\tilde{E'}$ is obtained by a band sum from $D'$ and $E'$. 
Since $s_{\{\partial \tilde{D'},\partial \tilde{E'}\}}(\partial D)=w_{\{\partial \tilde{D'},\partial \tilde{E'}\}}(\partial D)=x^{p_i}_{\partial \tilde{D'}}x^{p_j}_{\partial \tilde{E'}}$, we have $w_{\{\partial D',\partial E'\}}(\partial D)=(x_{\partial D'}x_{\partial E'})^{p_i}x^{p_j}_{\partial E'}$ or $(x_{\partial E'}x_{\partial D'})^{p_i}x^{-p_j}_{\partial E'}$ (, or $w_{\{\partial D',\partial E'\}}(\partial D)=x^{-p_i}_{\partial D'}(x_{\partial D'}x_{\partial E'})^{p_j}$ or $x^{p_i}_{\partial D'}(x_{\partial E'}x_{\partial D'})^{p_j}$ when $\tilde{D'}=E'$) by Remark \ref{changeunderbs}. 
Note that $w_{\{\partial D',\partial E'\}}(\partial D)$ does not have both of ``$x^{2}_{\partial D'}$ or $x^{-2}_{\partial D'}$'' and ``$x^{2}_{\partial E'}$ or $x^{-2}_{\partial E'}$''. 
This implies that $s_{\{\partial D',\partial E'\}}(\partial D)$ does not have both of ``$x^{2}_{\partial D'}$ or $x^{-2}_{\partial D'}$'' and ``$x^{2}_{\partial E'}$ or $x^{-2}_{\partial E'}$'' since otherwise, $s_{\{\partial D',\partial E'\}}(\partial D)=w_{\{\partial D',\partial E'\}}(\partial D)$ by Observation \ref{ck2} and $w_{\{\partial D',\partial E'\}}(\partial D)$ would have both. 
In the diagram of $\partial D' \cup \partial E'$ in $\mathcal{P}_{\{\partial D_{L},\partial E\}}$, if $D$ intersects with $D_{L}\cup E$, then $s_{\{\partial D',\partial E'\}}(\partial D)$ has $x^{2}_{\partial E'}$ or $x^{-2}_{\partial E'}$, see the left of Figure \ref{lr}. 
In the diagram of $\partial D' \cup \partial E'$ in $\mathcal{P}_{\{\partial D_{R},\partial E\}}$, if $D$ intersects with $D_{R}\cup E$, then $s_{\{\partial D',\partial E'\}}(\partial D)$ has $x^{2}_{\partial D'}$ or $x^{-2}_{\partial D'}$, see the right of Figure \ref{lr}. 
Hence $D$ is disjoint from $D_{L}\cup E$ or $D_{R}\cup E$. 
Such $D$ is $D_{L}$, $D_{R}$, $E$, a result of a band sum of $D_{L}$ and $E$, or a result of a band sum of $D_{R}$ and $E$. 
However, in any cases, $s_{\{\partial D',\partial E'\}}(\partial D)$ would have both of ``$x^{2}_{\partial D'}$ or $x^{-2}_{\partial D'}$'' and ``$x^{2}_{\partial E'}$ or $x^{-2}_{\partial E'}$'', see Figure \ref{lr}. 
We get a contradiction.

\begin{figure}[htbp]
 \begin{center}
  \includegraphics[width=120mm]{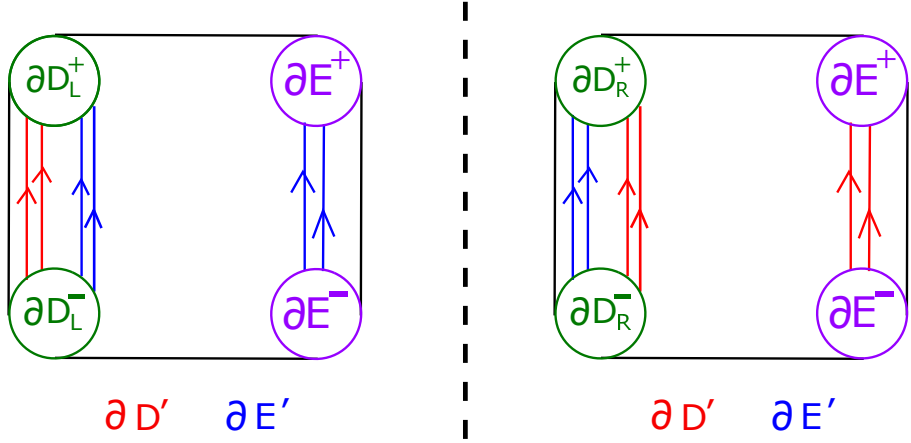}
 \end{center}
 \caption{A part of the diagrams of $\partial D'\cup \partial E'$ on $\mathcal{P}_{\{\partial D_{L},\partial E\}}$ and $\mathcal{P}_{\{\partial D_{R},\partial E\}}$}
 \label{lr}
\end{figure}

(3.2)The case where $\tilde{D'}$ is neither $D'$ nor $E'$.\\
Then $\tilde{E'}$ is $D'$ or $E'$ and $\tilde{D'}$ is obtained by a band sum from $D'$ and $E'$. 
We get a contradiction by an argument similar to that in (3.1)\\

(3.3)The case where $(\tilde{D'},\tilde{E'})$ is $(D',E')$ or $(E',D')$.\\
Note that $(\tilde{D'},\tilde{E'})=(E',D')$ may happen when $p_{i}=p_{j}$. 
Consider the diagram of $\partial D'\cup \partial E'$ on $\mathcal{P}_{\{\partial D_{L},\partial E\}}$ (see Figure \ref{p_dle2}). 
Note that there are arcs, at least two consecutive parallel arcs with the same orientation coming from $\partial D'$ connecting $\partial D^{+}_{L}$ and $\partial D^{-}_{L}$, at least two consecutive parallel arcs with the same orientation coming from $\partial E'$ connecting $\partial D^{+}_{L}$ and $\partial D^{-}_{L}$, and $p_{j}-|q'_{j}|$ consecutive parallel arcs with the same orientation coming from $\partial E'$ connecting $\partial E^{+}$ and $\partial E^{-}$. 
We call these arcs {\it the vertical arcs}. \\

(3.3.1)The case where $D$ intersects with $D_{L}\cup E$.\\
Take an outermost disk of $D$ cut by $D\cap (D_{L}\cup E)$. 
This corresponds to an arc on $\mathcal{P}_{\{\partial D_{L},\partial E\}}$ both of whose endpoints are on one boundary component. 
By Observation \ref{cc-1}, there exist a subarc of $\partial D$ both of whose endpoints are on one boundary component, which is ``the reverse sign'' of the boundary mentioned before. 
The corresponding arcs of these subarcs are disjoint from the vertical arcs since $\partial D$ intersect with $\partial E'$ ,(which is one of $\{\partial \tilde{D'},\partial \tilde{E'}\}$) in the same sign. 
If the outermost disk is cut by $E$, there are a pair of intersection points of $\partial D\cap \partial D'$ and a pair of intersection points of $\partial D \cap \partial E'$ such that each point of these are separated by each other. This contradict the condition of $D$ and $\{\tilde{D'},\tilde{E'}\}$ under $(\tilde{D'},\tilde{E'})=(D',E')$ or $(E',D')$. 
See the left of Figure \ref{p_dle4}. 
If the outermost disk is cut by $D_{L}$, $|\partial D\cap \partial E'|\geq 2(p_{j}-|q'_{j}|)>p_{j}$.
See the right of the Figure \ref{p_dle4}. 
If $E'=\tilde{E'}$, this leads a contradiction, and if $E'=\tilde{D'}$, this also leads a contradiction since $p_{i}=p_{j}$ in this case. 

\begin{figure}[htbp]
 \begin{center}
  \includegraphics[width=120mm]{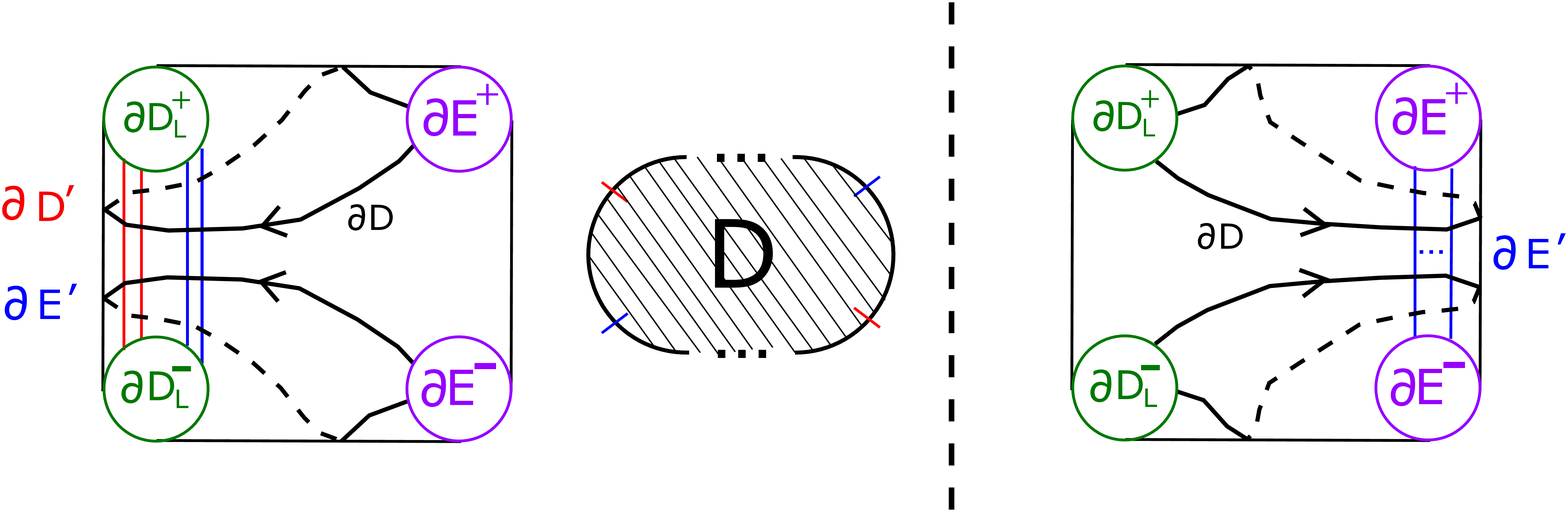}
 \end{center}
 \caption{A part of the diagrams of $\partial D'\cup \partial E'$ on $\mathcal{P}_{\{\partial D_{L},\partial E\}}$ and $\partial D$}
 \label{p_dle4}
\end{figure}

(3.3.2)The case where $D$ is disjoint from $D_{L}\cup E$.\\
There are three possibilities for $D$: ``$D$ is a result of a band sum of $D_{L}$ and $E$'', ``$D=D_{L}$'', and ``$D=E$''. 
Suppose that $D$ is a result of a band sum of $D_{L}$ and $E$. 
Then the band is disjoint from the vertical arcs since $\partial D$ intersect with $\partial E'$ ,(which is one of $\{\tilde{D'},\tilde{E'}\}$) in the same sign. 
Note that in the diagram of $\partial D'\cup \partial E'$ on $\mathcal{P}_{\{\partial D_{L},\partial E\}}$, there are parallel $|q'_{j}|(p_{k}-1)$ arcs (not consecutive) coming from $\partial E'$ connecting $\partial D^{+}_{L}$ and $\partial D^{-}_{L}$, and parallel $p_{j}-|q'_{j}|$ arcs coming from $\partial E'$ connecting $\partial E^{+}$ and $\partial E^{-}$. 
Thus $|\partial D \cap \partial E'|\geq |q'_{j}|(p_{k}-1)+p_{j}-|q'_{j}|>p_{j}$. 
This leads a contradiction. 
Suppose that $D=D_{L}$. 
In the useful left standard diagram of any cases, we can find ${\rm min}\{p_{k}, |q_{k}|\}$ parallel oriented arcs connecting one of the boundary components coming from $\partial D'$ and one of the boundary components coming from $\partial E'$ in the diagram of $\partial D_{L}$ in $\mathcal{P}_{\{\partial D',\partial E'\}}$. 
This means that $s_{\{\partial D',\partial E'\}}(\partial D_{L})$ or $s_{\{\partial D',\partial E'\}}(\partial D_{L})$ has ${\rm min}\{p_{k}, |q_{k}|\}$-times $x_{\partial D'}x_{\partial E'}$ (, $x_{\partial D'}x^{-1}_{\partial E'}$, $x^{-1}_{\partial D'}x_{\partial E'}$, or $x^{-1}_{\partial D'}x^{-1}_{\partial E'}$). 
This contradicts to the condition of $D$ and $\{ \tilde{D'},\tilde{E'}\}$ under $D=D_{L}$ and $\{ \tilde{D'},\tilde{E'}\}=\{D',E'\}$.\\

By the above, the only possible case is $D=E$ and $(\tilde{D'},\tilde{E'})=(D',E')$ or $(E',D')$. This completes the proof.

\subsection{A proof of Claim \ref{claim2}}
Suppose that there exists an element $f$ of $\mathcal{G}(H_{\{i,j\}})$ such that exchange $D'$ and $E'$ (of the useful standard diagram). 
By Claim \ref{claim1}, all elements of $\mathcal{G}(H_{\{i,j\}})$ fix $E$ as a set. 
By comparing the intersection numbers of $\partial E\cap \partial D'$ and $\partial E \cap \partial E'$, $p_{i}=p_{j}$ must holds. 
By composing $h$ in Subsection \ref{invs} if necessary, we assume that $f$ preserves $E$ with the orientation. 
We consider the action of $f$ on the set of the intersection points of $\partial E$ and $\partial D'\cup \partial E'$, see Figure \ref{int_we}. 
In this figure, $(n)$ denotes a number $n$ modulo $p_i$, and $[m]$ denotes a number $m$ modulo $p_j$. 
Since $f$ preserve the orientation of $E$, the intersections $(0)$ and $(q'_{i})$ are mapped to $[p_{j}-1]$ and $[p_{j}-q'_{i}-1]$, respectively. 
By following $\partial D'$ from $(0)$ toward the upper side of $E$ (in Figure \ref{int_we}), we reach at $(q'_{i})$. 
By following $\partial E'$ from $[p_{j}-1]$ toward the upper side of $E$ (in Figure \ref{int_we}), we reach at $[p_{j}-q'_{j}-1]$. 
This implies that $q'_{i}\equiv q'_{j}$ mod $p_{i}(=p_{j})$. 
Recall that $2|q'_{i}|<p_{i}$ and $2|q'_{j}|<p_j$. 
Thus $q'_{i}=q'_{j}$. 
Recall also that $q'_{i}=q'_{j}$ if and only if $q_{i}\equiv q_{j}$ mod $p_{i}(=p_{j})$. \\
When $p_{i}=p_{j}$ and $q_{i}\equiv q_{j}$ mod $p_{i}$, there is actually an element, $\iota$ in Subsection \ref{invs} for example, of $\mathcal{G}(H_{\{i,j\}})$ exchanging $D'$ and $E'$.

\begin{figure}[htbp]
 \begin{center}
  \includegraphics[width=100mm]{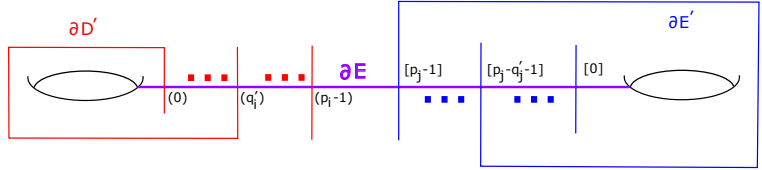}
 \end{center}
 \caption{Intersections of $\partial E$ and $\partial D'\cup \partial E'$}
 \label{int_we}
\end{figure}

\subsection{A proof of Claim \ref{claim3}}
Suppose that $g$ is an element of $\mathcal{G}(H_{\{i,j\}})$ mapping $E$ to $E$ preserving the orientation and mapping $D'$ and $E'$ to themselves. 
Since $g$ preserves the orientation of $E$, $g$ fixes the each intersection point in Figure \ref{int_we}. 
This implies that $g$ fixes the orientations of $D'$ and $E'$. 
Consider the domains of $S\setminus (\partial D'\cup \partial E'\cup \partial E)$, see Figure \ref{domains}. 
In the right of Figure \ref{domains}, all domains are put so that the front side of $S$ is front side of this paper. 
Each of two octagons is mapped to itself by $g$ since $g$ can be regarded as an orientation preserving self homeomorphism of $S$. 
By this, we see that each domain is mapped to itself by $g$. 
Moreover, on the boundary of each domain, $g$ is isotopic to the identity. 
By the Alexander's trick, we see that $g$ is isotopic to the identity as a self homeomorphism of $S$.

\begin{figure}[htbp]
 \begin{center}
  \includegraphics[width=80mm]{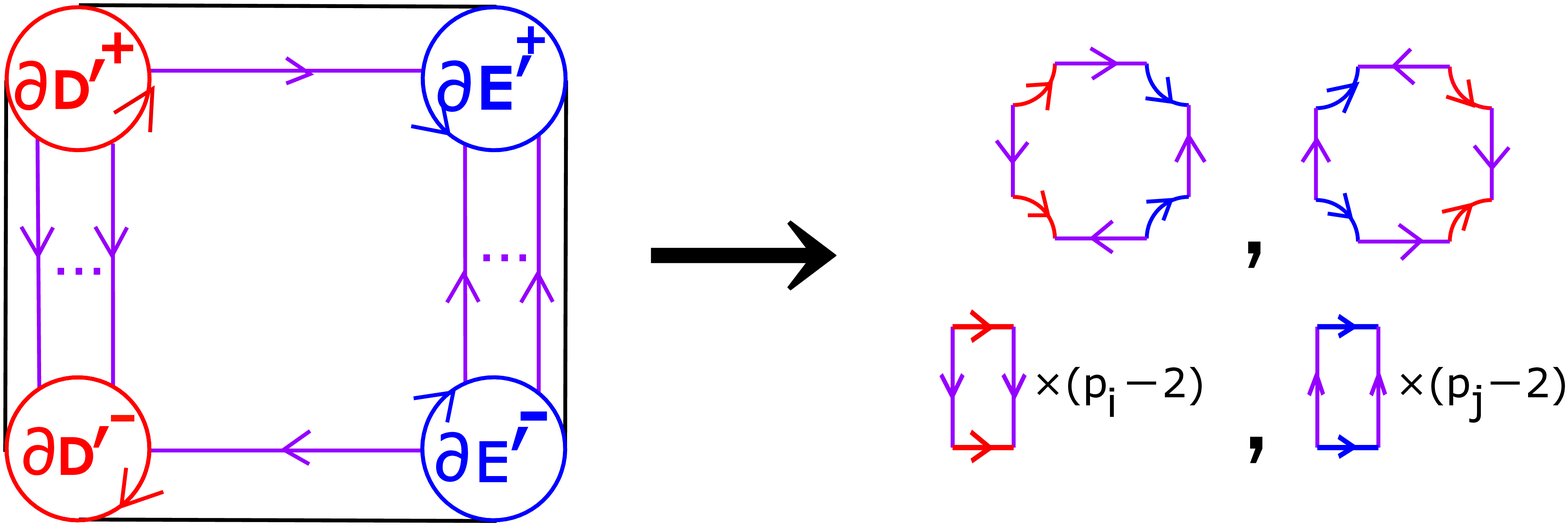}
 \end{center}
 \caption{Domains of $S\setminus (\partial D'\cup \partial E'\cup \partial E)$}
 \label{domains}
\end{figure}

\vspace{0.5cm}
\ GRADUATE SCHOOL OF MATHEMATICAL SCIENCES, THE UNIVERSITY OF TOKYO, 3-8-1 KOMABA, MEGURO--KU, TOKYO, 153-8914, JAPAN\\
\ \ E-mail address: \texttt{sekinonozomu@g.ecc.u-tokyo.ac.jp}

\end{document}